\documentclass[11pt,twoside]{article}
\usepackage[utf8]{inputenc}
\usepackage{geometry}
\usepackage{amsmath,amsthm,amssymb,chngcntr,setspace,titlesec}
\usepackage[all]{xy}

\usepackage{fancyhdr}
\newtheorem{thm}{Theorem}[section]
\newtheorem{lemma}[thm]{Lemma}

\newtheorem{prop}[thm]{Proposition}
\newtheorem{cor}[thm]{Corollary}
\newtheorem{defn}[thm]{Definition}
\newtheorem{examp}[thm]{Example}

\newtheorem{rmk}[thm]{Remark}
\numberwithin{equation}{subsection}
\title{Weil-\'etale Cohomology and Special Values of L-functions}
\author{Minh-Hoang Tran}
\begin{document}
\date{}
\maketitle

\begin{abstract}
We construct the Weil-\'etale cohomology and Euler characteristics for a subclass of the class of $\mathbb{Z}$-constructible sheaves on an open subscheme of the spectrum of the ring of integers of a number field. Then we show that the special value of an Artin L-function of toric type at zero is given by the Weil-\'etale Euler characteristic of an appropriate  $\mathbb{Z}$-constructible sheaf up to signs. As applications of our result, we will prove a formula for the special value of the L-function of an algebraic torus at zero which is similar to Ono's Tamagawa Number Formula.
\end{abstract}
\section{Introduction}

Let $K$ be a number field with ring of integers $O_K$ and Galois group $G_K$. Let $S$ be a finite set of places containing the set of archimedian places $S_{\infty}$ of $K$. Let $U=Spec(O_{K,S})$ and $j: Spec(K)\to U$ be the inclusion of the generic point. By a torsion free discrete $G_K$-module of finite type, we mean a torsion free, finitely generated abelian group with a continuous action by $G_K$. There are two main aims of this paper.
\begin{enumerate}
\item The first is to construct the Weil-\'etale cohomology $H^n_W(U,\mathcal{F})$ and Euler characteristic $\chi_U(\mathcal{F})$ for any strongly-$\mathbb{Z}$-constructible sheaf $\mathcal{F}$ on $U$ (see definition $\ref{strong_const}$) with the following properties
\begin{itemize}
\item If $\mathcal{F}$ is constructible then $\chi_U(\mathcal{F})=1$.
\item If $\pi' :V \to U$ is a finite morphism and $\mathcal{F}$ is strongly-$\mathbb{Z}$-constructible on $V$ then $\pi'_{*}\mathcal{F}$ is strongly-$\mathbb{Z}$-constructible on $U$  and $\chi_V(\mathcal{F})=\chi_U(\pi'_{*}\mathcal{F})$. 
\item $\chi_U$ is multiplicative with respect to some special classses of short exact sequences of strongly-$\mathbb{Z}$-constructible sheaves on $U$.
\end{itemize}
\item The second is to prove the following theorem.
\begin{thm}\label{main}
Let $M$ be a torsion free discrete $G_K$-module of finite type. Suppose $L_S(M,s)$ is the Artin L-function associated with the representation $M\otimes_{\mathbb{Z}}\mathbb{C}$ of $G_K$ modulo the local factors at $S$. Then   $\mathrm{ord}_{s=0}L_S(M,s)= 
\mathrm{rank}_{\mathbb{Z}} Hom_{U}(j_{*}M,\mathbb{G}_m) $ and 
\[ L^{*}_S(M,0)=\pm \chi_U(j_{*}M).\]
\end{thm}
\end{enumerate}
  
The fact that $L_S^{*}(M,0)$ is related to the Euler characteristic of $j_{*}M$ was established by Bienenfeld and Lichtenbaum and our proof is based on the techniques they developed in $\cite{Lic75}$ and $\cite{BL}$. However, their Euler characteristic is constructed from the Artin-Verdier cohomology and is different from our Weil-\'etale Euler characteristic. The Weil-\'etale cohomology in this paper is not the same as the one constructed by Lichtenbaum in $\cite{Lic09a}$ but rather is based on his ideas in $\cite{Lic09b}$ and $\cite{Lic14}$ together with some modification to treat the non-totally imaginary number field case. We hope that the results in this paper provide evidence for Lichtenbaum's general philosophy namely : the special values of L-functions are given by the Weil-\'etale Euler characteristics of appropriate sheaves. As applications, we will prove the following theorem which is similar to Ono's Tamagawa number formula $\cite{Ono63}$.
\begin{thm}\label{main_2}
Let $T$ be an algebraic torus defined over a number field $K$ with character group $\hat{T}$.
Let $h_{T,S}$, $R_{T,S}$ and $w_T$ be the $S$-class number, the $S$-regulator and the number of roots of unity of $T$.
Let $\mathbb{III}^1(T)$ be the Tate-Shafarevich group. Then $\mathrm{ord}_{s=0}L_S(\hat{T},s)=\mathrm{rank}_{\mathbb{Z}}T(O_{K,S})$ and
\begin{eqnarray}
L^{*}_S(\hat{T},0)&=&\pm \frac{h_{T,S}R_{T,S}}{w_T}\frac{[\mathbb{III}^1(T)]}{[H^{1}(K,\hat{T})]} \prod_{v\in S}{[H^1(K_v,\hat{T})]}
\prod_{v \notin S}[H^0(\hat{\mathbb{Z}},H^1(I_v,\hat{T}))].
\end{eqnarray}
\end{thm}
The structure of the paper is as follows. We construct the Weil-\'etale cohomology and the regulator pairing in section 2. In sections 3, we discuss the main properties of  strongly-$\mathbb{Z}$-constructible sheaves.  In section 4, we construct the Weil-\'etale Euler characteristics and show that they have the properties listed above. In section 5, we prove our main results namely Theorems $\ref{main}$ and $\ref{main_2}$ and illustrate them using the norm tori of quadratic extensions over $\mathbb{Q}$. Finally, we have an appendix containing the results about determinants of exact sequences and orders of torsion subgroups used in this paper. The readers are advised to skim through the appendix before reading section 4.

\medskip

\textbf{Acknowledgment:} This paper is part of my PhD thesis written at Brown University. I would like to thank my advisor Professor Stephen Lichtenbaum for suggesting this problem to me. It would have been impossible for me to complete this project without his guidance and encouragement. Part of this work was written when I was a member of the SFB 1085 Higher Invariant Research Group at University of Regensburg. I would like to thank Professor Guido Kings for his support and my friend Yigeng Zhao for many helpful conversations.
\section{The Weil-\'etale Cohomology of $\mathbb{Z}$-Constructible Sheaves}
\subsection{The Weil-\'etale Complexes}
We fix the following notations for the whole paper. Let $K$ be a number field with ring of integers $O_K$ and Galois group $G_K$. Let $S_{\infty}$ be the set of all archimedean places of $K$ and $S$ be a finite set of places of $K$ containing $S_{\infty}$. Let $U=Spec(O_{K,S})$ and $j: Spec(K)\to U$ be the inclusion of the generic point.
In this section, we define the Weil-\'etale complex for $\mathbb{Z}$-constructible sheaves on $U$ following the ideas of Lichtenbaum $\cite{Lic14}$. First, we recall the definition of $\mathbb{Z}$-constructible sheaves from $\cite[\mbox{page 146}]{Mil06}$.
\begin{defn}\label{defn_construct}
	A sheaf $\mathcal{F}$ on $U$ is $\mathbb{Z}$-constructible if 
		\begin{enumerate}
			\item there exists an open dense subscheme $V$ of $X$ and a finite \'etale covering 
	$ V' \to V$ such that the restriction of $\mathcal{F}$ to $V'$ is a constant sheaf defined by a finitely generated abelian group, 
			\item for any point $p$ outside $V$, the stalk $\mathcal{F}_{\bar{p}}$ is a finitely generated abelian group.
		\end{enumerate}
We say that $\mathcal{F}$ is constructible if in the definition above the restriction of $\mathcal{F}$ to $V'$ is a  constant sheaf defined by a finite abelian group and for any point $p$ outside $V$, the stalk $\mathcal{F}_{\bar{p}}$ is finite.
\end{defn}

\begin{examp}
\begin{enumerate}
\item Any constant sheaf defined by a finitely-generated abelian groups.
\item Let $M$ be a discrete $G_K$-module, then $j_{*}M$ is a $\mathbb{Z}$-constructible sheaf. Furthermore, if $M$ is finite then $j_{*}M$ is constructible.
\end{enumerate}
 \end{examp}
 We need the cohomology with compact support constructed by Milne in \cite[\mbox{page 165}]{Mil06}. Let $\mathcal{F}$ be an \'etale sheaf on $U$. Let $C^{\circ}(\mathcal{F})$ be the canonical $\check{C}$ech complex defined in \cite[\mbox{page 145}]{Mil06} and $C^{\circ}(U,\mathcal{F})$ be $\Gamma(U,C^{\circ}(\mathcal{F}))$ the complex of its global sections. Then under the assumption on  $U$, $C^{\circ}(U,\mathcal{F})\simeq R\Gamma_{et}(U,\mathcal{F})$.
 
For each prime $v$ of $K$, let $\mathcal{F}_v$ be the discrete $G_{K_v}$-module corresponding to the pull-back of $\mathcal{F}$ to $Spec(K_v)$. Let $C^{\circ}(\mathcal{F}_v)$ be the standard inhomogeneous resolution of $\mathcal{F}_v$. If $v$ is a real prime then let $S^{\circ}(\mathcal{F}_v)$ be the  complete standard inhomogeneous resolution of $\mathcal{F}_v$, if not we define $S^{\circ}(\mathcal{F}_v)$ to be $C^{\circ}(\mathcal{F}_v)$.
Then there is a morphism of complexes 
\[ u : C^{\circ}(U,\mathcal{F}) \to \prod_{v\in S} S^{\circ}(\mathcal{F}_v).\]
We write $R\Gamma_c(U,\mathcal{F})$ for the translate $C^{\circ}(u)[-1]$ of the mapping cone of $u$ and $H^n_c(U,\mathcal{F}):= h^n(R\Gamma_c(U,\mathcal{F}))$ is defined as the cohomology with compact support of $\mathcal{F}$. We have the following long exact sequence
\begin{equation}\label{compact_coh}
... \to H^n_c(U,\mathcal{F}) \to H^n_{et}(U,\mathcal{F}) \to \prod_{v\in S_{\infty}} H^n_T(K_v,\mathcal{F}_v) \oplus \prod_{v\in S-S_{\infty}}H^n(K_v,\mathcal{F}_v) \to H^{n+1}_c(U,\mathcal{F})  \to ... 
\end{equation}
where $H^n_T(K_v,\mathcal{F}_v)$ is the Tate cohomology of the finite group $G_{K_v}=G(\mathbb{C}/K_v)$.

\begin{defn}\label{Weil_etale_defn}
Let  $\mathcal{F}$ be a $\mathbb{Z}$-constructible sheaf on $U$, the Weil-\'etale complex is defined as
\[ R\Gamma_W(U,\mathcal{F}) := 
\tau_{\leq 1}R\Gamma_{c}(U,\mathcal{F}) \oplus
 \tau_{\geq 2}R\mathrm{Hom}_{\mathbb{Z}}(R\mathrm{Hom}_{U}(\mathcal{F},\mathbb{G}_m[-1]),\mathbb{Z}[-3]) \]
where $\tau_{\leq n}$ and $\tau_{\geq n}$ are the truncation functors defined in $\cite[\mbox{1.2.7}]{Weibel94}$. This is an object in the derived category of abelian groups. The Weil-\'etale cohomology are defined by 
\[ H^n_W(U,\mathcal{F}):=h^n(R\Gamma_W(U,\mathcal{F})).\]
\end{defn}
For an abelian group $A$, we write $A^D:=Hom_{\mathbb{Z}}(A,\mathbb{Q}/\mathbb{Z})$ and $A^{*}:=Hom_{\mathbb{Z}}(A,\mathbb{Z})$.
\begin{prop}\label{Weil_F}
The Weil-\'etale cohomology of $\mathcal{F}$ satisfy
\begin{equation}
H^n_W(U,\mathcal{F})= 
       \left\{
				\begin{array}{ll}
					H^n_{c}(U,\mathcal{F})&  \mbox{$n\leq 1$}\\
					{Hom}_U(\mathcal{F},\mathbb{G}_m)_{tor}^D & \mbox{$n=3$}  \\
					0 & \mbox{$ n > 3$}.
				\end{array}
			   \right.
\end{equation}
\[ 0 \to Ext^{1}_{U}(\mathcal{F},\mathbb{G}_m)_{tor}^D \to H^{2}_W(U,\mathcal{F}) \to Hom_{U}(\mathcal{F},\mathbb{G}_m)^{*} \to 0. \]
\end{prop}
\begin{proof}
From the definition of $R\Gamma_W(U,\mathcal{F})$, $H^n_W(U,\mathcal{F})\simeq H^n_{c}(U,\mathcal{F})$   for $n\leq 1$.
For $n\geq 2$, from $\cite[\mbox{exercise 3.6.1}]{Weibel94}$, there is an exact sequence 
\[ 0 \to Ext^{3-n}_{U}(\mathcal{F},\mathbb{G}_m)_{tor}^D \to H^{n}_W(U,\mathcal{F}) \to Ext^{2-n}_{U}(\mathcal{F},\mathbb{G}_m)^{*} \to 0 .\]
Hence, $H^{n}_W(U,\mathcal{F})=0$ for $n\geq 4$
and $H^{3}_W(U,\mathcal{F}) \simeq Hom_{U}(\mathcal{F},\mathbb{G}_m)_{tor}^{D}$. 
\end{proof}
\begin{prop}\label{long_exact_Weil}
Suppose we have an exact sequence of $\mathbb{Z}$-constructible sheaves
\begin{equation}\label{short_exact_etale}
 0 \to \mathcal{F}_1 \to \mathcal{F}_2 \to \mathcal{F}_3 \to 0.
\end{equation}
Assume further that $Ext^1_U(\mathcal{F}_i,\mathbb{G}_m)$ is finite for all $i$. 
Then we have a long exact sequence of Weil-\'etale cohomology
\begin{equation}\label{long_Weil_a}
 ...\to H^n_{W}(U,\mathcal{F}_1) \to H^n_{W}(U,\mathcal{F}_2)  \to ... \to H^3_{W}(U,\mathcal{F}_2) \to H^3_{W}(U,\mathcal{F}_3) \to  0.
\end{equation}
\end{prop}

\begin{proof}
As $R\mathrm{Hom}_U(-,\mathbb{G}_m[-1])$ and $R\mathrm{Hom}_{\mathbb{Z}}(-,\mathbb{Z}[-3])$ are exact functors, we have a distinguished triangle 
\begin{multline}\label{long_exact_Weil_seq1}
 R\mathrm{Hom}_{\mathbb{Z}}(R\mathrm{Hom}_U(\mathcal{F}_1,\mathbb{G}_m[-1]),\mathbb{Z}[-3]) \to 
R\mathrm{Hom}_{\mathbb{Z}}(R\mathrm{Hom}_U(\mathcal{F}_2,\mathbb{G}_m[-1]),\mathbb{Z}[-3]) \to \\
\to R\mathrm{Hom}_{\mathbb{Z}}(R\mathrm{Hom}_U(\mathcal{F}_3,\mathbb{G}_m[-1]),\mathbb{Z}[-3]) \to 
R\mathrm{Hom}_{\mathbb{Z}}(R\mathrm{Hom}_U(\mathcal{F}_1,\mathbb{G}_m[-1]),\mathbb{Z}[-3])[1].
\end{multline}
The long exact sequence of cohomology corresponding to ($\ref{long_exact_Weil_seq1}$) yields 
\begin{multline}\label{long_exact_Weil_seq2} 
H^2_W(U,\mathcal{F}_1) \to H^2_W(U,\mathcal{F}_2)  
\to H^2_W(U,\mathcal{F}_3) \to 
H^3_W(U,\mathcal{F}_1) \to H^3_W(U,\mathcal{F}_2) \to H^3_W(U,\mathcal{F}_3) \to 0.
\end{multline}
As $Ext^1_U(\mathcal{F}_i,\mathbb{G}_m)$ is finite,  
$Hom_{U}(\mathcal{F}_1,\mathbb{G}_m)^{*} \to Hom_{U}(\mathcal{F}_2,\mathbb{G}_m)^{*} $ is injective. Moreover, by the Artin-Verdier Duality \cite[\mbox{II.3.1}]{Mil06}
$Ext^1_{U}(\mathcal{F}_i,\mathbb{G}_m)_{tor}^D\simeq H^2_{c}(U,\mathcal{F}_i)$.
Applying the Snake Lemma to 
\[ \xymatrixrowsep{0.2in}
\xymatrix{ 0 \ar[r] & H^2_{c}(U,\mathcal{F}_1)    \ar[d] \ar[r]  & H^2_{W}(U,\mathcal{F}_1) \ar[d] \ar[r]   & Hom_{U}(\mathcal{F}_1,\mathbb{G}_m)^{*}  \ar[d] \ar[r] &  0 \\
  0 \ar[r] & H^2_{c}(U,\mathcal{F}_2)  \ar[r] &   H^2_{W}(U,\mathcal{F}_2) \ar[r]  &  Hom_{U}(\mathcal{F}_2,\mathbb{G}_m)^{*}       \ar[r]    &  0 }\]
  we deduce
\[ \ker(H^2_c(U,\mathcal{F}_1) \to H^2_c(U,\mathcal{F}_2)) \simeq 
\ker(H^2_W(U,\mathcal{F}_1) \to H^2_W(U,\mathcal{F}_2)).\]

Since $H^n_W(U,\mathcal{F})\simeq H^n_c(U,\mathcal{F})$ for $n\leq 1$, combining ($\ref{long_exact_Weil_seq2}$) with the exact sequence of cohomology with compact support corresponding to ($\ref{short_exact_etale}$), we obtain ($\ref{long_Weil_a}$).
\end{proof}

\subsection{The Regulator Pairing}
We want to define a pairing for every \'etale sheaf on $U$ such that it generalizes the construction of the $S$-regulator of a number field when the sheaf is $\mathbb{Z}$. 
For each place $v$ of $K$, let $j_v$ be the map $Spec(K_v) \to U$. To ease notations, we make the following definition.
\begin{defn}\label{defn_FBFDB}
For a sheaf $\mathcal{F}$ on $U$, we define the sheaf $\mathcal{F}_S$ to be
\[ \mathcal{F}_S:=\prod_{v \in S} (j_v)_{*}(j_v)^{*}\mathcal{F}. \]
 
\end{defn} 
There is a natural map $\mathcal{F}\to \mathcal{F}_S$ obtained by taking the direct sum over all $v$ in $S$ of the map $\mathcal{F} \to (j_v)_{*}(j_v)^{*}\mathcal{F}$. If we write $\mathcal{F}_v$ for the discrete $G_{K_v}$-module corresponding to $(j_v)^{*}\mathcal{F}$ then
\[ H^0_{et}(U,\mathcal{F}_S)\simeq \prod_{v\in S}H^0(K_v,\mathcal{F}_v).\]
By the product formula, the following map $\Lambda_K$ is well-defined.
\[ \Lambda_K : \frac{H^0_{et}(U,(\mathbb{G}_m)_S)}{H^0_{et}(U,\mathbb{G}_m)}=\frac{\prod_{v \in S} K_v^{*}}{O_{K,S}^{*}} \to \mathbb{R} \qquad 
\textbf{u}=(u_1,...,u_v) \mapsto \sum_{v \in S} \log|u_v|_v.\] 

\begin{defn}\label{defn_pairing} 
Let $\mathcal{F}$ be an \'etale sheaf on $U$. The regulator pairing for $\mathcal{F}$ 
\begin{equation}\label{eqn_pairing}
 \langle \cdot , \cdot \rangle_{\mathcal{F}} :  \frac{H^0_{et}(U,\mathcal{F}_S)}{H^0_{et}(U,\mathcal{F})} \times Hom_{U}(\mathcal{F},\mathbb{G}_m) \to \mathbb{R} 
\end{equation}
is defined as  follows.
Let $\alpha$ and $\phi$ be elements of ${H^0_{et}(U,\mathcal{F}_S)}/{H^0_{et}(U,\mathcal{F})}$ and $Hom_{U}(\mathcal{F},\mathbb{G}_m)$. By functoriality, $\phi$ induces a map 
\[  \phi_S : \frac{H^0_{et}(U,\mathcal{F}_S)}{H^0_{et}(U,\mathcal{F})} \to  
\frac{H^0_{et}(U,(\mathbb{G}_m)_S)}{H^0_{et}(U,\mathbb{G}_m)} 
= \frac{\prod_{v \in S} K_v^{*}}{O_{K,S}^{*}}. \]
Define 
$ \langle \alpha,\phi\rangle_{\mathcal{F}} := \Lambda_K(\phi_S(\alpha)) = 
\sum_{v \in S} \log|\phi_S(\alpha)_v|_v .$
\end{defn}

\begin{defn}\label{defn_regulator}
Suppose the pairing $\langle,\rangle_{\mathcal{F}}$ is non-degenerate modulo torsion. Choose bases $\{v_i\}$ and $\{u_j\}$ for the torsion free quotient groups of
$H^0_{et}(U,\mathcal{F}_S)/H^0_{et}(U,\mathcal{F})$ and $Hom_{U}(\mathcal{F},\mathbb{G}_m)$ respectively. Define $R(\mathcal{F}):=|\det(\langle v_i,u_j\rangle_{\mathcal{F}})|$ to be {the regulator of $\mathcal{F}$}. This definition does not depend on the choice of bases.
\end{defn}

\begin{examp}\label{ex_pairing}
\begin{enumerate}
\item If $\mathcal{F}$ is constructible then $Hom_{U}(\mathcal{F},\mathbb{G}_m)$ and $H^0_{et}(U,\mathcal{F}_S)$ are finite groups. Thus, the pairing is trivial and $R(\mathcal{F})=1$.
\item Consider the constant sheaf $\mathbb{Z}$ on $U$, the regulator pairing in this case is 
\[ {\left(\prod_{v \in S}\mathbb{Z}\right)}/{\mathbb{Z}} \times O_{K,S}^{*} \to \mathbb{R} 
:  \quad ((n_v)_{v\in S}, u) \mapsto \sum_{v\in S} n_v\log|u|_v.
\]
Let $\{u_1,...u_{[S]-1}\}$ be a $\mathbb{Z}$-basis for $O_{K,S}^{*}/\mu_K$. Let $\{e_1,..e_{[S]-1}\}$ be the standard basis for 
 $\prod_{v \in S}\mathbb{Z}/\mathbb{Z}$. Then $\langle e_v,u_j\rangle=\log|u_j|_{v}$. The determinant of the matrix $(\log|u_j|_{v})$ is the $S$-unit regulator $R_{K,S}$ of the number field $K$. 
Hence $R(\mathbb{Z})=R_{K,S}$.
\item Let $T$ be an algebraic torus over a number field $K$. Then the pairing \eqref{defn_pairing} for $j_{*}\hat{T}$ can be identified with the following paring of the torus $T$ :
\[  \frac{\prod_{v\in S}H^0(K_v,\hat{T})}{H^0(K,\hat{T})} \times T(O_{K,S})
\to \mathbb{R} :\quad ((\chi_v)_{v\in S},x) \mapsto \sum_{v\in S} \log|\chi_v(x)|_v.\]
 In particular, it is non-degenerate modulo torsion. Moreover, the regulator $R({j_{*}\hat{T}})$ is the same as the regulator $R_{T,S}$ defined in $\cite{Ono61}$.
\end{enumerate}
\end{examp}

\section{Strongly-$\mathbb{Z}$-Constructible Sheaves}
\label{chapter_strongly_Z}
 \subsection{Definitions and Examples}
\begin{defn}\label{strong_const}
A $\mathbb{Z}$-constructible sheaf $\mathcal{F}$ on $U=Spec(O_{K,S})$ is called strongly-$\mathbb{Z}$-constructible if it satisfies the following conditions :
\begin{enumerate}
\item The map $ H^0_{et}(U,\mathcal{F}) \to H^0_{et}(U,\mathcal{F}_S) $ has finite kernel.
\item $H^1_{et}(U,\mathcal{F})$ and $Ext^1_{U}(\mathcal{F},\mathbb{G}_m)$ are finite abelian groups.
\item The regulator pairing $(\ref{eqn_pairing})$ is non-degenerate modulo torsion.
\end{enumerate}
\end{defn}

\begin{rmk}\label{Artin_Verdier_strongly}
Our definition is modeled after the definition of quasi-constructible sheaves of Bienenfeld and Lichtenbaum (cf. $\cite[\mbox{section 4}]{BL}$). 
\end{rmk}

\begin{examp}
\begin{enumerate}
\item Constant sheaves defined by finitely generated abelian groups.
\item Let $v$ be a closed point of $U$ and  $i : v \to U$ be the natural map. Let $M$ be a finite $\hat{\mathbb{Z}}$-module. Then $i_{*}M$ is strongly-$\mathbb{Z}$-constructible. A non-example would be $i_{*}\mathbb{Z}$. Indeed, the kernel of 
$H^0_{et}(U,i_{*}\mathbb{Z}) \to H^0_{et}(U,(i_{*}\mathbb{Z})_S)$ is isomorphic to $\mathbb{Z}$. 
\item Constructible sheaves.
\item Let $M$ be a torsion free discrete $G_K$-module of finite type. Then $j_{*}M$ is strongly-$\mathbb{Z}$-constructible (see Proposition $\ref{jM_strong_constr}$).
\end{enumerate}
\end{examp}

The following proposition is a direct consequence of Proposition $\ref{Weil_F}$ and Artin-Verdier duality.
\begin{prop}
Let $\mathcal{F}$ be a strongly-$\mathbb{Z}$-constructible sheaf on $U$. Then 
 \[H^{n}_{W}(U,\mathcal{F}) = 
       \left\{
				\begin{array}{ll}
					H^n_{c}(U,\mathcal{F}) & \mbox{$n\leq 1$} \\
					Hom_U(\mathcal{F},\mathbb{G}_m)_{tor}^D &  \mbox{$n=3$}\\
					0 &  \mbox{$n > 3$.}
				\end{array}
			   \right.
	            \]           
 \[ 0 \to H^2_{c}(U,\mathcal{F}) \to H^2_{W}(U,\mathcal{F}) \to Hom_U(\mathcal{F},\mathbb{G}_m)^{*} \to 0.\]
 In particular, if $\mathcal{F}$ is a constructible sheaf, then
$H^n_{W}(U,\mathcal{F})=H^n_c(U,\mathcal{F})$ for all $n$. 
\end{prop}
\subsection{Main Properties}
We study the main properties of strongly-$\mathbb{Z}$-constructible sheaves in this section. 
\begin{prop}\label{abel_constr}
Suppose we have an exact sequence of \'etale sheaves on $U$
\begin{equation}\label{abel_constr_seq1}
 0 \to \mathcal{F}_1 \to \mathcal{F}_2 \to \mathcal{F}_3 \to 0 
\end{equation}
where $\mathcal{F}_3$ is constructible. Then $\mathcal{F}_1$ is strongly-$\mathbb{Z}$-constructible if and only if 
$\mathcal{F}_2$ is strongly-$\mathbb{Z}$-constructible.
\end{prop}

\begin{proof}
Since $\mathcal{F}_3$ is constructible, by  $\cite[\mbox{page 146}]{Mil06}$,  $\mathcal{F}_1$ is $\mathbb{Z}$-constructible if and only if $\mathcal{F}_2$ is $\mathbb{Z}$-constructible. 

As $H^n_{et}(U,\mathcal{F}_3)$ and $Ext^n_{U}(\mathcal{F}_3,\mathbb{G}_m)$ are finite for all $n$, $H^n_{et}(U,\mathcal{F}_1)$ and $H^n_{et}(U,\mathcal{F}_2)$ differ only by finite groups and so do $Ext^1_{U}(\mathcal{F}_1,\mathbb{G}_m)$ and $Ext^1_{U}(\mathcal{F}_2,\mathbb{G}_m)$. Since $H^0_{et}(U,\mathcal{F}_{3,S})$ and $H^0_{et}(U,\mathcal{F}_3)$ are finite,
$ \left({H^0_{et}(U,\mathcal{F}_{1,S})}/{H^0_{et}(U,\mathcal{F}_1)}\right)_{\mathbb{R}} \simeq 
\left({H^0_{et}(U,\mathcal{F}_{2,S})}/{H^0_{et}(U,\mathcal{F}_2)}\right)_{\mathbb{R}} $.
Hence, conditions 1 and 2 of $\ref{strong_const}$ hold for $\mathcal{F}_1$ if and only if they hold for $\mathcal{F}_2$.

By functoriality, there is a commutative diagram 
\[ \xymatrixrowsep{0.2in}\xymatrix{ \left(\frac{H^0_{et}(U,\mathcal{F}_{1,S})}{H^0_{et}(U,\mathcal{F}_1)}\right)_{\mathbb{R}} \ar[d]^{\simeq} & \times & Hom_{U}(\mathcal{F}_1,\mathbb{G}_m)_{\mathbb{R}}  \ar[r] & \mathbb{R} \ar[d]^{id} \\
\left(\frac{H^0_{et}(U,\mathcal{F}_{2,S})}{H^0_{et}(U,\mathcal{F}_2)}\right)_{\mathbb{R}} & \times & Hom_{U}(\mathcal{F}_2,\mathbb{G}_m)_{\mathbb{R}} \ar[u]^{\simeq} \ar[r] & \mathbb{R} }
\]
As a result, condition 3 of $\ref{strong_const}$ holds for $\mathcal{F}_1$ if and only if it holds for $\mathcal{F}_2$.
\end{proof}
Next we want to show that strongly-$\mathbb{Z}$-constructible sheaves are stable under push-forward by a finite morphism.
Let $L/K$ be a finite Galois extension of number fields. Let $V$ be the normalization of $U=Spec(O_{K,S})$ in $L$. Then $V=Spec(O_{L,S'})$ where $S'$ is the set of places of $L$ lying over a place of $K$ in $S$. Let $\pi : Spec(L) \to Spec(K)$ and $ \pi': V \to U$ be the natural finite morphisms. 
\begin{lemma}\label{Mackey}
Let $\mathcal{F}$ be a sheaf on $V$. For each place $v$ of $K$ and each place $w$ of $L$ lying over $K$, 
let $j_{w} : Spec(L_w) \to V$ and $\pi_{w} : Spec(L_w) \to Spec(K_v)$ be the natural maps. 
Then \[ j_{v}^{*}\pi'_{*}\mathcal{F}\simeq \prod_{w|v}(\pi_w)_{*}j_{w}^{*}\mathcal{F}.\] 
\end{lemma}
\begin{proof}
We have the commutative diagram
\[\xymatrixrowsep{0.2in}
\xymatrix @C=1in{
\prod_{w|v}Spec(L_w) \ar[d]^{\prod_{w|v}\pi_w}  \ar[r]^{\prod_{w|v}j_w} &  Spec(O_{L,S'})  \ar[d]^{\pi'}  \\
                     Spec(K_v)   \ar[r]^{j_v} &  Spec(O_{K,S})  }
                     \]
Since  $ \pi'_{*}(j_w)_{*}=(j_v)_{*}(\pi_w)_{*}$ for $w|v$,
\begin{eqnarray*}
Hom_{K_v}(j_{v}^{*}\pi'_{*}\mathcal{F},\prod_{w|v}(\pi_w)_{*}j_{w}^{*}\mathcal{F})
 &\simeq & Hom_{U}(\pi'_{*}\mathcal{F},\prod_{w|v}(j_{v})_{*}(\pi_w)_{*}j_{w}^{*}\mathcal{F}) \\
&\simeq & Hom_{U}(\pi'_{*}\mathcal{F},\pi'_{*}\prod_{w|v}(j_{w})_{*}j_{w}^{*}\mathcal{F}) .
\end{eqnarray*}
Thus, the adjoint map $\mathcal{F} \to \prod_{w|v}(j_{w})_{*}j_{w}^{*}\mathcal{F}$ induces a canonical map $j_{v}^{*}\pi'_{*}\mathcal{F} \to \prod_{w|v}(\pi_w)_{*}j_{w}^{*}\mathcal{F}$. Let $\eta_K=Spec(K)$, $\eta_v=Spec(K_v)$ and similarly for $\eta_L$ and $\eta_w$. Then
\[ (j_{v}^{*}\pi_{*}\mathcal{F})_{\eta_v}=\mathcal{F}_{\eta_L}^{[L:K]}=\left(\prod_{w|v}(\pi_w)_{*}j_{w}^{*}\mathcal{F}\right)_{\eta_v}. \]
Therefore, $j_{v}^{*}\pi'_{*}\mathcal{F}\simeq \prod_{w|v}(\pi_w)_{*}j_{w}^{*}\mathcal{F}$. 
\end{proof}

\begin{lemma}\label{pre_finite_inv}
Let $\mathcal{F}$ be a $\mathbb{Z}$-constructible sheaf on $V$. 
Then the following hold 
\begin{enumerate}
\item The norm map induces a natural isomorphism $Nm: Ext_{V}^{n}(\mathcal{F},\mathbb{G}_m) \to Ext_{U}^{n}(\pi'_{*}\mathcal{F},\mathbb{G}_m)$. 
\item There is a natural isomorphism $H^{n}_{W}(U,\pi'_{*}\mathcal{F}) \simeq H^{n}_{W}(V,\mathcal{F})$.
\item The sheaf $(\pi'_{*}\mathcal{F})_S$ is isomorphic to $\pi'_{*}(\mathcal{F}_S)$. In particular, $ H^0_{et}(U,(\pi'_{*}\mathcal{F})_S) \simeq H^0_{et}(V,\mathcal{F}_S)$.
\end{enumerate}
\end{lemma}
\begin{proof}
\begin{enumerate}
\item
We begin by describing the map $Nm$. The norm map $N_{L/K}$ induces a morphism of sheaves $N_{L/K}:\pi'_{*}\mathbb{G}_{m,V} \to \mathbb{G}_{m,U}$. As $\pi'$ is a finite morphism, $\pi'_{*}$ is an exact functor $\cite[\mbox{II.3.6}]{Mil80}$. Therefore, $\pi'_{*}$ induces the map 
$Ext_{V}^{n}(\mathcal{F},\mathbb{G}_{m,V}) \to Ext_{U}^{n}(\pi'_{*}\mathcal{F},\pi'_{*}\mathbb{G}_{m,V})$. We define $Nm$ to be the composition of the following maps
\[  Ext_{V}^{n}(\mathcal{F},\mathbb{G}_{m,V}) \xrightarrow{\pi'_{*}} Ext_{U}^{n}(\pi'_{*}\mathcal{F},\pi'_{*}\mathbb{G}_{m,V})
\xrightarrow{N_{L/K}} Ext_{U}^{n}(\pi'_{*}\mathcal{F},\mathbb{G}_{m,U}). \]
The fact that $Nm$ is an isomorphism is proved in $\cite[\mbox{II.3.9}]{Mil06}$.
\item By \cite[\mbox{II.2.3}]{Mil06}, $H^{n}_{c}(U,\pi'_{*}\mathcal{F}) \simeq H^{n}_{c}(V,\mathcal{F})$. Therefore, $ H^{n}_{W}(U,\pi'_{*}\mathcal{F}) \simeq H^{n}_{W}(V,\mathcal{F}) $ for $n\leq 1$. In addition, by part 1, we have 
\[ H^{3}_{W}(U,\pi'_{*}\mathcal{F}) = Hom_{U}(\pi'_{*}\mathcal{F},\mathbb{G}_m)_{tor}^D 
\simeq Hom_{V}(\mathcal{F},\mathbb{G}_m)_{tor}^D  = H^{3}_{W}(V,\mathcal{F}).
\]
To prove $H^{2}_{W}(V,\mathcal{F}) \simeq H^{2}_{W}(U,\pi'_{*}\mathcal{F})$, we apply the 5-lemma to the following diagram 
\[ \xymatrixrowsep{0.2in}
\xymatrix{ 0 \ar[r] & Ext^1_{U}(\pi'_{*}\mathcal{F},\mathbb{G}_m)_{tor}^{D}    \ar[d]^{\simeq} \ar[r]  & H^2_{W}(U,\pi'_{*}\mathcal{F}) \ar[d] \ar[r]   & Hom_{U}(\pi'_{*}\mathcal{F},\mathbb{G}_m)^{*}  \ar[d]^{\simeq} \ar[r] &  0 \\
  0 \ar[r] & Ext^1_{V}(\mathcal{F},\mathbb{G}_m)_{tor}^{D}  \ar[r] &   H^2_{W}(V,\mathcal{F}) \ar[r]  &  Hom_{V}(\mathcal{F},\mathbb{G}_m)^{*}       \ar[r]    &  0 }\]
\item Recall from Lemma $\ref{Mackey}$ that for $w|v$, we have $ (j_v)_{*}(\pi_w)_{*}=\pi'_{*}(j_w)_{*}$.
In addition, $(j_v)^{*}\pi'_{*}\mathcal{F} \simeq \prod_{w|v}(\pi_w)_{*}(j_w)^{*}\mathcal{F}$. Hence,
\begin{eqnarray*}
(\pi'_{*}\mathcal{F})_S &=& \prod_{v \in S}(j_v)_{*}(j_v)^{*}\pi'_{*}\mathcal{F} 
\quad \simeq \prod_{v \in S}(j_v)_{*}\left(\prod_{w|v}(\pi_w)_{*}(j_w)^{*}\mathcal{F}\right)  \\
&=&  \prod_{w \in S'}(j_v)_{*}(\pi_w)_{*}(j_w)^{*}\mathcal{F}  \quad
= \prod_{w \in S'}\pi'_{*}(j_w)_{*}(j_w)^{*}\mathcal{F} = \pi'_{*}(\mathcal{F}_{S'}).
\end{eqnarray*}
Therefore,  
$ H^0_{et}(U,(\pi'_{*}\mathcal{F})_S) \simeq H^0_{et}(U,\pi'_{*}(\mathcal{F}_{S'}))\simeq H^0_{et}(V,\mathcal{F}_{S'})$.
\end{enumerate}
\end{proof}

\begin{prop}\label{finite_inv}
If $\mathcal{F}$ is strongly-$\mathbb{Z}$-constructible then so is $\pi'_{*}\mathcal{F}$ and $R(\pi'_{*}\mathcal{F})=R(\mathcal{F})$.
\end{prop}
\begin{proof}
As $\pi'_{*}$ preserves $\mathbb{Z}$-constructible sheaves
$\cite[\mbox{page 146}]{Mil06}$, $\pi'_{*}\mathcal{F}$ is $\mathbb{Z}$-constructible. From Lemma $\ref{pre_finite_inv}$, it remains to show that the regulator pairing of $\pi'_{*}\mathcal{F}$ is non-degenerate and $R(\pi'_{*}\mathcal{F})=R(\mathcal{F})$.
They will all follow once we prove the diagram below commutes
\begin{equation}\label{finite_inv_seq}
\xymatrixrowsep{0.2in}  
  \xymatrix{ \left(\frac{H^0_{et}(U,(\pi'_{*}\mathcal{F})_S)}{H^0_{et}(U,\pi'_{*}\mathcal{F})}\right) \ar[d]^{\psi}_{\simeq}  & \times & Hom_{U}(\pi'_{*}\mathcal{F},\mathbb{G}_m)   \ar[r]  & \mathbb{R} \ar[d]^{id}  \\
 \left(\frac{H^0_{et}(V,\mathcal{F}_{S'})}{H^0_{et}(V,\mathcal{F})}\right) & \times & 
 Hom_{V}(\mathcal{F},\mathbb{G}_m) \ar[u]^{Nm}_{\simeq}  \ar[r]  & \mathbb{R} }
\end{equation}
Let $\alpha$ and $\phi$ be elements of $H^0_{et}(U,(\pi'_{*}\mathcal{F})_{S})$ and $Hom_{V}(\mathcal{F},\mathbb{G}_m)$. We need to show
\begin{equation}\label{finite_inv_eq2}
\Lambda_L \circ \phi_S(\psi(\alpha))=\Lambda_K \circ Nm(\phi)_S(\alpha).
\end{equation}
From Lemma $\ref{pre_finite_inv}$ , $Nm(\phi)=N_{L/K}\circ \pi'_{*}\phi$. Let us consider the following diagram
\begin{equation}\label{finite_inv_eq3}
 \xymatrixrowsep{0.2in}\xymatrixcolsep{5pc}\xymatrix {
 &   &    &    \\
\left(\frac{H^0_{et}(U,(\pi'_{*}\mathcal{F})_S)}{H^0_{et}(U,\pi'_{*}\mathcal{F})}\right)_{\mathbb{R}}  
 \ar[r]^-{(\pi'_{*}\phi)_S} \ar[d]^{\psi}   \ar@/^3pc/[rr]^-{Nm(\phi)_S}
 & \left(\frac{H^0_{et}(U,(\pi'_{*}\mathbb{G}_m)_S)}{H^0_{et}(U,\pi'_{*}\mathbb{G}_m)}\right)_{\mathbb{R}}  
 \ar[r]^-{(N_{L/K})_S} \ar[d]^{\psi}
& \left(\frac{H^0_{et}(U,(\mathbb{G}_m)_S)}{H^0_{et}(U,\mathbb{G}_m)}\right)_{\mathbb{R}}  
\ar[d]^-{\Lambda_K} \\
\left(\frac{H^0_{et}(V,\mathcal{F}_{S'})}{H^0_{et}(V,\mathcal{F})}\right)_{\mathbb{R}}
\ar[r]^-{\phi_{S'}}
& \left(\frac{H^0_{et}(V,(\mathbb{G}_m)_{S'})}{H^0_{et}(V,\mathbb{G}_m)}\right)_{\mathbb{R}} 
\ar[r]^-{\Lambda_L}  \ar[ur]^{N_{L/K}}
& \mathbb{R} 
}
\end{equation}
The left square of ($\ref{finite_inv_eq3}$) commutes by functoriality. It is not hard to see the upper triangle in the square on the right hand side is commutative. We shall prove that the lower triangle also commutes. Let $\beta=(\beta_w)_{w\in S'}$ be an element of $H^0_{et}(V,(\mathbb{G}_m)_S) \simeq \prod_{w \in S'} L_w^{*}$. Then 
\begin{eqnarray*}
\Lambda_K(N_{L/K}(\beta)) &=& \sum_{v \in S} \log|N_{L/K}(\beta)_v|_v 
= \sum_{v \in S} \sum_{w|v} \log|N_{L_w/K_v}(\beta_w)|_v  \\
&=& \sum_{v \in S} \sum_{w|v} \log|\beta_w|_w   
= \sum_{w \in S'}  \log|\beta_w|_w   
=  \Lambda_L(\beta).
\end{eqnarray*}
Therefore, diagram ($\ref{finite_inv_eq3}$) is commutative and from this we deduce equation ($\ref{finite_inv_eq2}$). 
As a result, diagram ($\ref{finite_inv_seq}$) commutes. Hence, the proposition is proved.
\end{proof}

\begin{cor}\label{cor_finite_inv}
$\pi'_{*}\mathbb{Z}$ is a strongly-$\mathbb{Z}$-constructible sheaf.
\end{cor}
\begin{prop}\label{ono_etale}
Let $M$ be a torsion free discrete $G_K$-module of finite type. Then there exist finitely many Galois extensions $\{K_{\mu}\}_{\mu}$, $\{K_{\lambda}\}_{\lambda}$ of $K$, a positive integer $n$ and a finite $G_K$-module $N$ such that 
if $\pi_{\mu}: Spec(K_{\mu})\to 
Spec(K)$ and $\pi_{\lambda}: Spec(K_{\lambda})\to 
Spec(K)$ are the natural maps then
we have the following exact sequence 
\begin{equation}\label{ono_etale_seq1}
0 \to M^{n}\oplus \prod_{\mu}(\pi_{\mu})_{*}\mathbb{Z} \to \prod_{\lambda}(\pi_{\lambda})_{*}\mathbb{Z} \to N \to 0.
\end{equation}
Furthermore, let $U=Spec(O_{K,S})$ where $S$ is a finite set of places of $K$ containing $S_{\infty}$. Let $V_{\mu}$ and $V_{\lambda}$ be the normalization of $U$ in $Spec(K_{\mu})$ and $Spec(K_{\lambda})$. Let  $\pi_{\mu}': V_{\mu} \to U$ and $\pi'_{\lambda}: V_{\lambda}\to 
U$ be the natural maps. Then there exists a constructible sheaf $\mathcal{Q}$ such that
\begin{equation}\label{ono_etale_seq2}
0 \to (j_{*}M)^{n}\oplus \prod_{\mu}(\pi_{\mu}')_{*}\mathbb{Z} \to \prod_{\lambda}(\pi_{\lambda}')_{*}\mathbb{Z} \to \mathcal{Q} \to 0.
\end{equation}
\end{prop}
\begin{proof}
The existence of ($\ref{ono_etale_seq1}$) is precisely $\cite[\mbox{1.5.1}]{Ono61}$.  
Let $P_1=\prod_{\mu}(\pi_{\mu})_{*}\mathbb{Z}$ and $P_2=\prod_{\lambda}(\pi_{\lambda})_{*}\mathbb{Z}$. Then  $j_{*}P_1=\prod_{\mu}(\pi_{\mu}')_{*}\mathbb{Z}$ and $j_{*}P_2=\prod_{\lambda}(\pi_{\lambda}')_{*}\mathbb{Z}$.
By applying $j_{*}$ to ($\ref{ono_etale_seq1}$), we obtain the exact sequence 
\[ 0 \to j_{*}M^n\oplus \prod_{\mu}(\pi_{\mu}')_{*}\mathbb{Z} \to \prod_{\lambda}(\pi_{\lambda}')_{*}\mathbb{Z} \to \mathcal{Q} \to 0\]
where $\mathcal{Q}$ is a subsheaf of $j_{*}N$. As $N$ is finite, $j_{*}N$ is constructible and so is $\mathcal{Q}$. 
\end{proof}
\begin{prop}\label{jM_strong_constr}
Let $M$ be a torsion free discrete $G_K$-module of finite type. Then $j_{*}M$ is a strongly-$\mathbb{Z}$-constructible sheaf.
\end{prop}
\begin{proof}
Consider sequence ($\ref{ono_etale_seq2}$) of Proposition $\ref{ono_etale}$.
As $\mathcal{Q}$ is constructible and $\prod_{\lambda}(\pi_{\lambda}')_{*}\mathbb{Z}$ is strongly-$\mathbb{Z}$-constructible, $(j_{*}M)^{n}\oplus \prod_{\mu}(\pi_{\mu}')_{*}\mathbb{Z}$ is strongly-$\mathbb{Z}$-constructible by Proposition $\ref{abel_constr}$. Since $\prod_{\mu}(\pi_{\mu}')_{*}\mathbb{Z}$ is strongly-$\mathbb{Z}$-constructible, so is  $j_{*}M$.
\end{proof}

\section{Euler Characteristics of Strongly-$\mathbb{Z}$-Constructible Sheaves}
\label{chapter_euler_strongly_Z}

\subsection{Construction} 
Let $\mathcal{F}$ be an \'etale sheaf on $U$. We use the notations of section 2. 
Composing the natural maps 
\[ R\Gamma_{W}(U,\mathcal{F}) \to \tau_{\leq 1} R\Gamma_{c}(U,\mathcal{F}) \to R\Gamma_{c}(U,\mathcal{F}) \quad \& \quad  R\Gamma_{c}(U,\mathcal{F}) \to C^{\circ}(U,\mathcal{F}) \to \prod_{v\in S_{\infty}}S^{\circ}(\mathcal{F}_v) \]
yields the map
\begin{equation}\label{defn_cF}
R\Gamma_W(U,\mathcal{F}) \xrightarrow{} 
\prod_{v\in S_{\infty}}S^{\circ}(\mathcal{F}_v).
\end{equation}
\begin{defn}\label{defn_dF}
We define the complex $D_{\mathcal{F}}$ by the translate of the mapping cone of \eqref{defn_cF}
\[ D_{\mathcal{F}} := 
\left[R\Gamma_W(U,\mathcal{F}) \xrightarrow{} 
\prod_{v\in S_{\infty}}S^{\circ}(\mathcal{F}_v) \right][-1]. \]
\end{defn}
 
 \begin{prop}\label{Hn_dF}
Let $\mathcal{F}$ be an \'etale sheaf on $U$. Then $H^n(D_\mathcal{F})$ satisfy 
\begin{multline}\label{Hn_dF_seq1}
0 \to H^0(D_\mathcal{F}) \to H^0_{c}(U,\mathcal{F}) \to \prod_{v\in S_{\infty}}H^0(K_v,\mathcal{F}_v)  \to H^1(D_\mathcal{F}) \to H^1_{c}(U,\mathcal{F}) \to \\ 
\to \prod_{v\in S_{\infty}}H^1(K_v,\mathcal{F}_v)  \to H^2(D_\mathcal{F})  
\to H^2_{W}(U,\mathcal{F}) 
\to \prod_{v\in S_{\infty}}H^2(K_v,\mathcal{F}_v) \to \\ \to H^3(D_\mathcal{F})
\to Hom_U(\mathcal{F},\mathbb{G}_m)_{tor}^D  
\to \prod_{v\in S_{\infty}}H^3(K_v,\mathcal{F}_v) \to H^4(D_\mathcal{F})
 \to 0.
\end{multline}
Moreover, $H^{n}(D_\mathcal{F})\simeq \prod_{v\in S_{\infty}}H^{n-1}_T(K_v,\mathcal{F}_v)$ for $n\notin \{0,1,..,4\}$. 
 \end{prop}
 \begin{proof}
There is a distinguished triangle
 \[ D_\mathcal{F} \to R\Gamma_W(U,\mathcal{F}) \to \prod_{v\in S_{\infty}}S^{\circ}(\mathcal{F}_v) \to D_\mathcal{F}[1]. \]
 The long exact sequence of cohomology  yields $\eqref{Hn_dF_seq1}$
 and $H^{n}(D_{\mathcal{F}}) \simeq \prod_{v\in S_{\infty}}H^{n-1}_T(K_v,\mathcal{F}_v)$ for $n\notin \{0,1,..,4\}$. The lemma then follows from Proposition $\ref{Weil_F}$.
 \end{proof}
 
 \begin{prop}\label{prop_theta_iso}
 Let $\mathcal{F}$ be a strongly-$\mathbb{Z}$-constructible sheaf on $U$. Then $H^n(D_\mathcal{F})$ is finite for $n\neq 1,2$ and  we can construct an isomorphism $\Theta(\mathcal{F}) :  H^1(D_\mathcal{F})_{\mathbb{R}} \to H^2(D_\mathcal{F})_{\mathbb{R}}$.
 \end{prop}
 
 \begin{proof}
 Clearly, $H^n(D_\mathcal{F})$ is finite for $n\notin \{0,1,2\}$. It is not hard to see that condition 1 of \ref{strong_const} is equivalent to the $H^0(D_\mathcal{F})$ is finite. 
 Consider the following diagram
 \begin{equation}\label{theta_iso_3x3}
 \xymatrix{
 &  0  \ar[d] &  0 \ar[d]  &  0 \ar[d] &  \\
0 \ar[r] &  \left(\frac{H^0_{et}(U,\mathcal{F})}{H^0_{c}(U,\mathcal{F})}\right)_{\mathbb{R}}
\ar[r]  \ar[d]  &   \underset{v\in S-S_{\infty}}{\prod}H^0(K_v,\mathcal{F}_v)_{\mathbb{R}} \ar[r]\ar[d]
&  H^1_c(U,\mathcal{F})_{\mathbb{R}} \ar[r]\ar[d]^{id} & 0  & (\mathcal{E}_1)  \\ 
0 \ar[r] &  \left(\frac{ \underset{v\in S-S_{\infty}}{\prod}H^0(K_v,\mathcal{F}_v)}{H^0_{c}(U,\mathcal{F})}\right)_{\mathbb{R}}
\ar[r]  \ar[d]  &   H^1(D_{\mathcal{F}})_{\mathbb{R}} \ar[r]\ar[d]
&  H^1_c(U,\mathcal{F})_{\mathbb{R}} \ar[r]\ar[d] & 0  & (\mathcal{E}_2)  \\ 
0 \ar[r] & 
\left(\frac{ \underset{v\in S-S_{\infty}}{\prod}H^0(K_v,\mathcal{F}_v)}{H^0_{et}(U,\mathcal{F})}\right)_{\mathbb{R}}
\ar[r]^{id}  \ar[d]  & \left(\frac{ \underset{v\in S-S_{\infty}}{\prod}H^0(K_v,\mathcal{F}_v)}{H^0_{et}(U,\mathcal{F})}\right)_{\mathbb{R}}   \ar[r]\ar[d]
&  0 \ar[r]\ar[d] & 0  \\ 
 &  0 &  0   &  0  &   \\
 &  (\mathcal{E}_3)  & (\mathcal{E}_0)
 }
 \end{equation}
 Note that the exact sequences $(\mathcal{E}_1)$,  $(\mathcal{E}_2)$ are induced from \eqref{compact_coh} and \eqref{Hn_dF_seq1} respectively, whereas   $(\mathcal{E}_3)$ is a canonical exact sequence. By choosing sections of $(\mathcal{E}_1)$ and  $(\mathcal{E}_2)$ we can construct  $(\mathcal{E}_0)$ such that \eqref{theta_iso_3x3} is commutative.
 Next, we consider the following diagram
 \begin{equation}\label{theta_iso_2x3}
  \xymatrix{
0 \ar[r] &  \underset{v\in S-S_{\infty}}{\prod}H^0(K_v,\mathcal{F}_v)_{\mathbb{R}}
\ar[r]  \ar[d]^{id}  &   H^1(D_{\mathcal{F}})_{\mathbb{R}} \ar[r]\ar[d]^{\Theta_1}
&  \left(\frac{ \underset{v\in S_{\infty}}{\prod}H^0(K_v,\mathcal{F}_v)}{H^0_{et}(U,\mathcal{F})}\right)_{\mathbb{R}}  \ar[r]\ar[d]^{id} & 0  & (\mathcal{E}_0)  \\ 
0 \ar[r] & \underset{v\in S-S_{\infty}}{\prod}H^0(K_v,\mathcal{F}_v)_{\mathbb{R}}
\ar[r]   & \left(\frac{ \underset{v\in S}{\prod}H^0(K_v,\mathcal{F}_v)}{H^0_{et}(U,\mathcal{F})}\right)_{\mathbb{R}}   \ar[r]
&  \left(\frac{ \underset{v\in S_{\infty}}{\prod}H^0(K_v,\mathcal{F}_v)}{H^0_{et}(U,\mathcal{F})}\right)_{\mathbb{R}} \ar[r] & 0  & (\mathcal{E}_4) 
 }
 \end{equation}
 Note that $(\mathcal{E}_4)$ is canonical and $(\mathcal{E}_0)$ is taken from \eqref{theta_iso_3x3}.  
Again by choosing sections of  $(\mathcal{E}_4)$ and $(\mathcal{E}_0)$, we can construct an isomorphism (not canonically),
 \begin{equation}
 \Theta_1 : H^1(D_{\mathcal{F}})_{\mathbb{R}}  \xrightarrow{\simeq} 
 \left(\frac{ \underset{v\in S}{\prod}H^0(K_v,\mathcal{F}_v)}{H^0_{et}(U,\mathcal{F})}\right)_{\mathbb{R}}
 \end{equation}
 such that \eqref{theta_iso_2x3} commutes.
 From  \eqref{Hn_dF_seq1}, we have an isomorphism 
 \[ \Theta_2: H^2(D_{\mathcal{F}})_{\mathbb{R}}\xrightarrow{\simeq} H^2_W(U,\mathcal{F})_{\mathbb{R}} \xrightarrow{\simeq} Hom_U(\mathcal{F},\mathbb{G}_m)^{*}_{\mathbb{R}}.\]
 Since the regulator pairing \eqref{defn_pairing} is non-degenerate, there is an isomorphism 
 \[ \Theta_3 : \left(\frac{ \underset{v\in S}{\prod}H^0(K_v,\mathcal{F}_v)}{H^0_{et}(U,\mathcal{F})}\right)_{\mathbb{R}} \simeq Hom_U(\mathcal{F},\mathbb{G}_m)^{*}_{\mathbb{R}} .\]
 Let us define $\Theta(\mathcal{F}):=\Theta_2^{-1}\Theta_3\Theta_1$. Then $\Theta(\mathcal{F}) :  H^1(D_\mathcal{F})_{\mathbb{R}} \to H^2(D_\mathcal{F})_{\mathbb{R}}$ is an isomorphism. 
 \end{proof}
  With respect to integral bases and subject to the condition that the diagrams \eqref{theta_iso_3x3} and \eqref{theta_iso_2x3} are commutative, the determinant of $\Theta(\mathcal{F})$ does not depend on the choices of sections of $(\mathcal{E}_i)$ and is well-defined. Therefore, we can make the following definition. 
\begin{defn}\label{euler_defn1}
Let $\mathcal{F}$ be a strongly-$\mathbb{Z}$-constructible sheaf  on $U$.
For each $v\in S_{\infty}$, let $h(G_v,\mathcal{F}_v)$ be the Herbrand quotient of $\mathcal{F}_v$ with respect to the cyclic group $G_v=G(\mathbb{C}/K_v)$.
 We define the Euler characteristic $\chi_U(\mathcal{F})$ by 
\[ {\chi}_U(\mathcal{F}):= \prod_{n=0}^5[H^n(D_{\mathcal{F}})_{tor}]^{(-1)^{n}}|\det(\Theta(\mathcal{F}))| \prod_{v\in S_{\infty}}h(G_v,\mathcal{F}_v)^{3}\]
where $\det(\Theta(\mathcal{F}))$  is computed with respect to integral bases.  We also define the rank of $\mathcal{F}$ to be
\[ E_U(\mathcal{F}):=\sum_{n\in \mathbb{Z}}(-1)^n n.\mathrm{rank}_{\mathbb{Z}}H^n(D_{\mathcal{F}}).\]
If $U=X=Spec(O_K)$ then we simply write $\chi(\mathcal{F})$ and $E(\mathcal{F})$ instead of $\chi_X(\mathcal{F})$ and $E_X(\mathcal{F})$  .
\end{defn}
We will compute $\chi_U(\mathcal{F})$ in the next theorem. This is a rather long and tedious calculation. The readers are advised to refer to the appendix for the facts we need about determinants of exact sequences and orders of torsion subgroups. 
 \begin{thm}\label{thm_chi1_formula}
Let $\mathcal{F}$  be a strongly-$\mathbb{Z}$-constructible sheaf on $U$. 
Let $\Psi^n(\mathcal{F})$ be the map $H^n_{et}(U,\mathcal{F}) \to \prod_{v\in S}H^n(K_v,\mathcal{F}_v)$ and $\delta(\mathcal{F})$ be the map
\[ \underset{v\in S}{\prod}H^0(K_v,\mathcal{F}_v) \to
{ \underset{v\in S}{\prod}H^0(K_v,\mathcal{F}_v)}/{H^0_{et}(U,\mathcal{F})}.
\]
Then $E_U(\mathcal{F})=\mathrm{rank}_{\mathbb{Z}}Hom_U(\mathcal{F},\mathbb{G}_m)$ and
\begin{equation}\label{eqn_thm_chi1}
{\chi}_U(\mathcal{F})=
 \frac{[H^0_{et}(U,\mathcal{F})_{tor}][Ext^1_U(\mathcal{F},\mathbb{G}_m)]R(\mathcal{F})}
{[Hom_U(\mathcal{F},\mathbb{G}_m)^D_{tor}][\ker(\Psi^1(\mathcal{F}))][\mathrm{cok}(\delta(\mathcal{F})_{tor})]\prod_{v\in S}[H^0(K_v,\mathcal{F}_v)_{tor}]} 
\end{equation}
 \end{thm}
 \begin{proof}
 Since $H^n(D_{\mathcal{F}})$ is finite for $n\neq 1,2$ and $\mathrm{rank}_{\mathbb{Z}}H^n(D_{\mathcal{F}})=\mathrm{rank}_{\mathbb{Z}}Hom_U(\mathcal{F},\mathbb{G}_m)$ for $n=1,2$, the first part of the theorem is clear. For the second part,
 we need to compute $\det(\Theta_i)$ for $i=1,2,3$. Clearly, $\det(\Theta_3)=R(\mathcal{F})$. Next we compute $\det(\Theta_2)$. From \eqref{Hn_dF_seq1}, we have
 \[ 0 \to Q_2 \to H^2(D_{\mathcal{F}}) \to H^2_W(U,\mathcal{F}) \to Q_3 \to 0\]
 where $Q_2$ is the cokernel of the map 
 \[ H^1_c(U,\mathcal{F}) \to \prod_{v\in S_{\infty}}H^1(K_v,\mathcal{F}_v)\]
 and $Q_3$ is the kernel of 
 \[ \prod_{v\in S_{\infty}}H^2(K_v,\mathcal{F}_v) \to H^3(D_{\mathcal{F}}) .\]
 Then by Proposition \ref{det_tor}
 \begin{eqnarray}
 \det(\Theta_2)&=& \frac{[H^2(D_{\mathcal{F}})_{tor}][Q_3]}{[Q_2][H^2_W(U,\mathcal{F})_{tor}]} 
 = \frac{[H^2(D_{\mathcal{F}})_{tor}][H^4(D_{\mathcal{F}})_{tor}][Hom_U(\mathcal{F},\mathbb{G}_m)^D_{tor}]}{[H^3(D_{\mathcal{F}})_{tor}][Q_2][H^2_c(U,\mathcal{F})]
  \prod_{v\in S_{\infty}}h(G_v,\mathcal{F}_v)}
.
 \end{eqnarray}
 Now we compute $\det(\Theta_1)$. From \eqref{theta_iso_2x3} and \eqref{theta_iso_3x3} and Lemma $\ref{det_3x3}$, we have 
 \[ \det(\Theta_1)=\frac{\nu(\mathcal{E}_4)}{\nu(\mathcal{E}_0)}=
 \frac{\nu(\mathcal{E}_1)\nu(\mathcal{E}_4)}{\nu(\mathcal{E}_2)\nu(\mathcal{E}_3)}\]
 where $\nu(\mathcal{E}_i)$ is the determinant of $(\mathcal{E}_i)$ with respect to integral bases (see the appendix for details).
 Hence, we need to compute $\nu(\mathcal{E}_i)$ for $i=1,..,4$. We will do so in the next lemmas. 
 \end{proof}
\begin{lemma}\label{lemma_v1v2}
With notations as in Theorem $\ref{thm_chi1_formula}$, we have 
\begin{eqnarray*}
\nu(\mathcal{E}_1)&=& \frac{\left[\left(\frac{H^0_{et}(U,\mathcal{F})}{H^0_{c}(U,\mathcal{F})}\right)_{tor}\right][H^1_{c}(U,\mathcal{F})_{tor}]}{\prod_{v\in S_{\infty}}[H^0_T(K_v,\mathcal{F}_v)]\prod_{v\in S-S_{\infty}}[H^0(K_v,\mathcal{F}_v)_{tor}][\ker(\Psi^1_U(\mathcal{F}))]} \\
\nu(\mathcal{E}_2)&=& \frac{\left[\left(\frac{\prod_{v\in S_{\infty}}H^0(K_v,\mathcal{F}_v)}{H^0_{c}(U,\mathcal{F})}\right)_{tor}\right][H^1_{c}(U,\mathcal{F})_{tor}]}
{[H^1(D_{\mathcal{F}})_{tor}][Q_1]} 
\end{eqnarray*}
 where $Q_1$ is the image of the map 
 \[ H^1_c(U,\mathcal{F}) \to \prod_{v\in S_{\infty}}H^1(K_v,\mathcal{F}_v).\]
\end{lemma} 
 \begin{proof}
 Note that $(\mathcal{E}_1)$ and $(\mathcal{E}_2)$ are induced by the exact sequences
 \begin{equation*}
  0 \to  \frac{H^0_{et}(U,\mathcal{F})}{H^0_{c}(U,\mathcal{F})} \to 
  \prod_{v\in S_{\infty}}H^0_T(K_v,\mathcal{F}_v)\oplus \prod_{v\in S-S_{\infty}}H^0(K_v,\mathcal{F}_v) \to H^1_{c}(U,\mathcal{F}) \to \ker(\Psi^1_U(\mathcal{F})) \to 0,
\end{equation*}  
\begin{equation*}
  0 \to  \frac{\prod_{v\in S_{\infty}}H^0(K_v,\mathcal{F}_v)}{H^0_{c}(U,\mathcal{F})} \to 
 H^1(D_\mathcal{F}) \to H^1_{c}(U,\mathcal{F}) \to Q_1 \to 0.
\end{equation*}
Then the lemma follows by just applying Proposition  \ref{det_tor}. 
 \end{proof}
 
 \begin{lemma}\label{lemma_v3v4}
 Let $\beta$ be the map $H^0_{et}(U,\mathcal{F}) \to \prod_{v\in S_{\infty}}H^0(K_v,\mathcal{F}_v)$. Then 
 \begin{eqnarray*}
\nu(\mathcal{E}_3)&=& \frac{[H^0(D_{\mathcal{F}})_{tor}]
\left[\left(\frac{H^0_{et}(U,\mathcal{F})}{H^0_{c}(U,\mathcal{F})}\right)_{tor}\right]\left[\left(\frac{\prod_{v\in S_{\infty}}H^0(K_v,\mathcal{F}_v)}{H^0_{et}(U,\mathcal{F})}\right)_{tor}\right]}
{\prod_{v\in S_{\infty}}[H^{-1}_T(K_v,\mathcal{F}_v)][\ker(\beta)]\left[\left(\frac{\prod_{v\in S_{\infty}}H^0(K_v,\mathcal{F}_v)}{H^0_{c}(U,\mathcal{F})}\right)_{tor}\right]} \\
\nu(\mathcal{E}_4)&=& \frac{[H^0_{et}(U,\mathcal{F})_{tor}]
\left[\left(\frac{\prod_{v\in S_{\infty}}H^0(K_v,\mathcal{F}_v)}{H^0_{et}(U,\mathcal{F})}\right)_{tor}\right]}
{[\ker(\beta)]\prod_{v\in S_{\infty}}[H^0(K_v,\mathcal{F}_v)_{tor}][\mathrm{cok}(\delta(\mathcal{F})_{tor})]}.
\end{eqnarray*}
 \end{lemma}
 
 \begin{proof}
 We consider the following composition of maps
\[  \xymatrixrowsep{0.6in}\xymatrix{  H^0_c(U,\mathcal{F}) \ar[r]^{\alpha} & 
H^0_{et}(U,\mathcal{F}) \ar[r]^{\Psi^0(\mathcal{F})} \ar@/^3pc/[rr]^-{\beta} &
\prod_{v\in S}H^0(K_v,\mathcal{F}_v) \ar[r]^{\pi}  &
 \prod_{v\in S_{\infty}}H^0(K_v,\mathcal{F}_v).
}
\] 
 From \eqref{compact_coh} and \eqref{Hn_dF_seq1}, 
 $\ker(\alpha)\simeq \prod_{v\in S_{\infty}} H^{-1}_T(K_v,\mathcal{F}_v)$ and $\ker(\beta\circ\alpha)\simeq H^0(D_{\mathcal{F}})$. Therefore, from the kernel-cokernel exact sequence
 \begin{multline}\label{eqn_v3v4_1}
 0 \to \prod_{v\in S_{\infty}} H^{-1}_T(K_v,\mathcal{F}_v) \to H^0(D_{\mathcal{F}})
 \to \ker(\beta) \to \frac{H^0_{et}(U,\mathcal{F})}{H^0_{c}(U,\mathcal{F})} \to \\
\to \frac{\prod_{v\in S_{\infty}}H^0(K_v,\mathcal{F}_v)}{H^0_{c}(U,\mathcal{F})} \to
 \frac{\prod_{v\in S_{\infty}}H^0(K_v,\mathcal{F}_v)}{H^0_{et}(U,\mathcal{F})} \to 0.
 \end{multline}
 Applying Proposition \ref{det_tor} to \eqref{eqn_v3v4_1} yields the formula for $\nu(\mathcal{E}_3)$. Again from the kernel-cokernel exact sequence,
 \begin{equation}\label{eqn_v3v4_2}
 0 \to \ker(\Psi^0(\mathcal{F})) \to \ker(\beta) \to 
 \prod_{v\in S-S_{\infty}}H^0(K_v,\mathcal{F}_v) \to 
 \frac{\prod_{v\in S}H^0(K_v,\mathcal{F}_v)}{H^0_{et}(U,\mathcal{F})} \to
 \frac{\prod_{v\in S_{\infty}}H^0(K_v,\mathcal{F}_v)}{H^0_{et}(U,\mathcal{F})} \to 0.
 \end{equation}
 Applying Proposition \ref{det_tor} to \eqref{eqn_v3v4_2} yields 
 \begin{equation}\label{eqn_v3v4_3}
 \nu(\mathcal{E}_4)= \frac{[\ker(\Psi^0_U(\mathcal{F}))]
\prod_{v\in S-S_{\infty}}[H^0(K_v,\mathcal{F}_v)_{tor}]
\left[\left(\frac{\prod_{v\in S_{\infty}}H^0(K_v,\mathcal{F}_v)}{H^0_{et}(U,\mathcal{F})}\right)_{tor}\right]}
{[\ker(\beta)]\left[\left(\frac{\prod_{v\in S}H^0(K_v,\mathcal{F}_v)}{H^0_{et}(U,\mathcal{F})}\right)_{tor}\right]}.
 \end{equation}
  Consider the exact sequence
 \begin{equation}\label{eqn_v3v4_4}
 0 \to \ker(\Psi^0(\mathcal{F})) \to H^0_{et}(U,\mathcal{F})
 \xrightarrow{\Psi^0(\mathcal{F})} \prod_{v\in S}H^0(K_v,\mathcal{F}_v) 
 \xrightarrow{\delta(\mathcal{F})}
 \frac{\prod_{v\in S}H^0(K_v,\mathcal{F}_v)}{H^0_{et}(U,\mathcal{F})} \to 0.
 \end{equation}
 Applying Lemma \ref{torsion_group} to \eqref{eqn_v3v4_4} yields
 \begin{equation}\label{eqn_v3v4_5}
 \frac{[\ker(\Psi^0(\mathcal{F}))]}{\left[\left( \frac{\prod_{v\in S_{\infty}}H^0(K_v,\mathcal{F}_v)}{H^0_{et}(U,\mathcal{F})}\right)_{tor}\right]}
 = \frac{[H^0_{et}(U,\mathcal{F})_{tor}]}
 {\prod_{v\in S}[H^0(K_v,\mathcal{F}_v)_{tor}][\mathrm{cok}(\delta(\mathcal{F})_{tor})]}.
 \end{equation}
 Combining \eqref{eqn_v3v4_3} and \eqref{eqn_v3v4_5} gives the formula for $\nu(\mathcal{E}_4)$.
 \end{proof}
 
\begin{proof}[Proof of Theorem \ref{thm_chi1_formula}]
From Lemmas \ref{lemma_v1v2} and \ref{lemma_v3v4}, 
\begin{equation}
\det(\Theta_1)=\frac{[H^1(D_{\mathcal{F}})_{tor}][H^0_{et}(U,\mathcal{F})_{tor}][Q_1]\prod_{v\in S_{\infty}}h(G_v,\mathcal{F}_v)^{-1}}
{[H^0(D_{\mathcal{F}})_{tor}]\prod_{v\in S}[H^0(K_v,\mathcal{F}_v)_{tor}][\mathrm{cok}(\delta(\mathcal{F})_{tor})][\ker(\Psi^1(\mathcal{F}))]}.
\end{equation}
Note that 
\[ [Q_1][Q_2]=\prod_{v\in S_{\infty}}[H^1(K_v,\mathcal{F}_v)]=[H^5(D_{\mathcal{F}})]
\prod_{v\in S_{\infty}}h(G_v,\mathcal{F}_v)^{-1}.\]
Then, putting everything together, we obtain \eqref{eqn_thm_chi1}.
\end{proof}
 
\subsection{Simple Computations}
As examples, we will compute a few Euler characteristics in this section. 
\begin{prop}\label{Z_Euler}
Let $h_S$, $R_S$ and $w$ be the $S$-class number, the $S$-regulator and the number of roots of unity of $K$ respectively. Then
$\chi_U(\mathbb{Z})={h_SR_S}/{w} $.
In particular, $\zeta_{K,S}^{*}(0)=-\chi(\mathbb{Z})$.
\end{prop}
\begin{proof}
 We have $H^1_{et}(U,\mathbb{Z})=0$ and $[Ext^1_U(\mathbb{Z},\mathbb{G}_m)]=[Pic(O_{K,S})]=h_S$. Clearly,  $\mathrm{cok}(\delta(\mathbb{Z})_{tor})=0$ and $\prod_{v\in S}H^0(K_v,\mathbb{Z})$ and $H^0_{et}(U,\mathbb{Z})$ are torsion-free.
In addition, $R(\mathbb{Z})=R_S$ and $[(O_{K,S}^{*})_{tor}]=w$. As a result,
${\chi}_U(\mathbb{Z})={h_SR_S}/{w}$.
Therefore, $\zeta_{K,S}^{*}(0) = -\chi_U(\mathbb{Z}) $ by \cite[\mbox{I.2.2}]{Tat84}.
\end{proof}
\begin{prop}\label{Zn_Euler}
Euler characteristics of finite constant sheaves are 1.
\end{prop}
\begin{proof}
 It suffices to prove this proposition for the constant sheaf $\mathbb{Z}/n$. For an abelian group $M$, we write $M[n]$ for the kernel of the multiplication-by-$n$ map.
We have $Hom_U(\mathbb{Z}/n,\mathbb{G}_m)\simeq \mu_n(K)$ and 
\[ 0 \to O_K^{*}/(O_K^{*})^{n}  \to Ext^1_U(\mathbb{Z}/n,\mathbb{G}_m) \to Pic(O_K)[n] \to 0 .\]
Observe that $[H^0_{et}(U,\mathbb{Z}/n)]=n$, $[H^0_{et}(U,(\mathbb{Z}/n)_S)]=n^{[S]}$ and $R(\mathbb{Z}/n)=1$. From Dirichlet's Unit Theorem, 
$[O_{K,S}^{*}/(O_{K,S}^{*})^{n}]=n^{[S]-1}\times [\mu(K)/\mu(K)^n]$. Also, $\mathrm{cok}(\delta(\mathbb{Z}/n)_{tor})=0$.  From \eqref{eqn_thm_chi1}, it remains to show 
$[\ker(\Psi^1(\mathbb{Z}/n))]=[Pic(O_{K,S})[n]]$. Indeed, we have the commutative diagram
\begin{equation*}
\xymatrix{
         &  &  0 \ar[d] & 0 \ar[d] \\
         &  & H^1_{et}(U,\mathbb{Z}/n) \ar[r]^{\Psi^1(\mathbb{Z}/n)} \ar[d] & \prod_{v\in S}H^1(K_v,\mathbb{Z}/n) \ar[d] \\
0 \ar[r] & H^2_c(U,\mathbb{Z}) \ar[r] \ar[d]^{n} & H^2_{et}(U,\mathbb{Z}) \ar[r]^{\Psi^2(\mathbb{Z})} \ar[d]^{n}
& \prod_{v\in S}H^2(K_v,\mathbb{Z}) \ar[d]^{n} \\
0\ar[r]  & H^2_c(U,\mathbb{Z})\ar[r]  & H^2_{et}(U,\mathbb{Z}) \ar[r]^{\Psi^2(\mathbb{Z})} 
& \prod_{v\in S}H^2(K_v,\mathbb{Z})  
}
\end{equation*}
By diagram chasing, $\ker(\Psi^1(\mathbb{Z}/n))\simeq H^2_c(U,\mathbb{Z})[n]$ which is isomorphic to $(Pic(O_{K,S})/n)^D$ by Artin-Verdier Duality. Hence, we are done.
\end{proof}
\begin{defn}\label{defn_negligible}
We say $\mathcal{F}$ is a negligible sheaf on $U$ if it has finite support and its stalks are finite everywhere. Note that negligible sheaves are constructible.
\end{defn}
\begin{prop}\label{Euler_negligible}
Euler characteristics of negligible sheaves are 1.
\end{prop}
\begin{proof} 
It is enough to prove this lemma for the sheaf $i_{*}M$ where $M$ is a finite $\hat{\mathbb{Z}}$-module. 
We have $H^{n}_{et}(U,i_{*}M) \simeq H^{n}(\hat{\mathbb{Z}},M)$ which is 0 for $n \geq 2$ 
$\cite[\mbox{page 189}]{Ser95}$. Moreover, $[H^0(\hat{\mathbb{Z}},M)]=[H^1(\hat{\mathbb{Z}},M)]$ for finite $M$  
$\cite[\mbox{page 32}]{Mil06}$. For any place $v$ and any $n$, $H^n(K_v,(i_{*}M)_v)=0$. Therefore ${\chi}_U(i_{*}M)=1$.
\end{proof}
Propositions \ref{Zn_Euler} and \ref{Euler_negligible} are special cases of the following proposition.
\begin{prop}\label{Euler_constructible}
Euler characteristics of constructible sheaves are 1. 
\end{prop}
\begin{proof}
Let $\mathcal{F}$ be a constructible sheaf on $U$. By \eqref{eqn_thm_chi1} and the Artin-Verdier Duality, we have 
\begin{equation}
{\chi}_U(\mathcal{F})=
\frac{[H^0_{et}(U,\mathcal{F})][H^2_c(U,\mathcal{F})]}
{[H^3_c(U,\mathcal{F})][\ker(\Psi^1(\mathcal{F}))]\prod_{v\in S}[H^0(K_v,\mathcal{F}_v)]} .
\end{equation}
From the exact sequence \eqref{compact_coh}
\begin{multline*}
0 \to \prod_{v\in S_{\infty}}H^{-1}_T(K_v,\mathcal{F}_v) \to H^0_{c}(U,\mathcal{F})
\to H^0_{et}(U,\mathcal{F}) \to \\
\to  \prod_{v\in S_{\infty}}H^{0}_T(K_v,\mathcal{F}_v)\oplus \prod_{v\in S-S_{\infty}}H^{0}(K_v,\mathcal{F}_v) \to H^1_{c}(U,\mathcal{F}) \to \ker(\Psi^1(\mathcal{F})) \to 0,
\end{multline*}
$\chi_U(\mathcal{F})$ can be rewritten as
\[
{\chi}_U(\mathcal{F})=
\frac{[H^0_{c}(U,\mathcal{F})][H^2_c(U,\mathcal{F})]}
{[H^1_c(U,\mathcal{F})][H^3_c(U,\mathcal{F})]\prod_{v\in S_{\infty}}[H^0(K_v,\mathcal{F}_v)]} =1\]
where the last equality follows from \cite[\mbox{II.2.13}]{Mil06}.
\end{proof}

\begin{prop}\label{euler_finite_inv}
Let $L$ be a finite Galois extension of $K$. Let $V=Spec(O_{L,S'})$ be the normalization of $U$ in $L$ and $\pi' : V \to U $ be the natural map. Let $\mathcal{F}$ be a strongly-$\mathbb{Z}$-constructible sheaf on $V$. Then $\pi'_{*}\mathcal{F}$ is a strongly-$\mathbb{Z}$-constructible sheaf on $U$. Moreover, $H^n(D_{\pi'_{*}\mathcal{F}}) \simeq H^n(D_{\mathcal{F}})$ and ${\chi}_U(\pi'_{*}\mathcal{F})={\chi}_V(\mathcal{F})$.
\end{prop}
\begin{proof}
From Proposition $\ref{finite_inv}$, $\pi'_{*}\mathcal{F}$ is a strongly-$\mathbb{Z}$-constructible sheaf and $R(\pi'_{*}\mathcal{F})=R(\mathcal{F})$. To show $H^n(D_{\pi'_{*}\mathcal{F}}) \simeq H^n(D_{\mathcal{F}})$, we only need to apply the 5-lemma to the following diagram
\[ \xymatrix @C=1.1pc{0 \ar[r]& H^0(D_{\pi'_{*}\mathcal{F}}) \ar[r] \ar[d] & H^0_{c}(U,\pi'_{*}\mathcal{F}) \ar[r] \ar[d]^{\simeq}&\prod_{v\in S_{\infty}}H^0(K_v,(\pi'_{*}\mathcal{F})_v) \ar[r] \ar[d]^{\simeq} & H^1(D_{\pi'_{*}\mathcal{F}}) \ar[r]\ar[d] & H^1_{c}(U,\pi'_{*}\mathcal{F}) \ar[r]\ar[d]^{\simeq} &  ...\\ 
0 \ar[r]& H^0(D_\mathcal{F}) \ar[r]  & H^0_{c}(V,\mathcal{F}) \ar[r] & \prod_{w\in S_{L,\infty}}H^0(L_w,\mathcal{F}_w) \ar[r] & H^1(D_\mathcal{F}) \ar[r] & H^1_{c}(V,\mathcal{F}) \ar[r]  &  ...\\ }
\]
where the rows are exact from Proposition $\ref{Hn_dF}$. Finally, each term in formula \eqref{eqn_thm_chi1} is invariant with respect to $\pi'_{*}$.
Therefore, ${\chi}_U(\pi'_{*}\mathcal{F})={\chi}_V(\mathcal{F})$.

\end{proof}

\begin{cor}\label{euler_piZ}
The sheaf $\pi_{*}\mathbb{Z}$ on $Spec(K)$  corresponds to the induced $G_K$-module ${Ind}_{G_L}^{G_K}(\mathbb{Z})$. If we write $\pi_{*}\mathbb{Z}$ for ${Ind}_{G_L}^{G_K}(\mathbb{Z})$ then
 $ L^{*}_S(\pi_{*}\mathbb{Z},0)=-\chi_U(\pi'_{*}\mathbb{Z}) $.
\end{cor}
\begin{proof}
By Propositions  $\ref{Z_Euler}$, $\ref{euler_finite_inv}$ and the fact that $ L_S(\pi_{*}\mathbb{Z},s)=\zeta_{L,S'}(s)$, we have 
\[ L_S^{*}(\pi_{*}\mathbb{Z},0)=\zeta^{*}_{L,S'}(0)=-\chi_V(\mathbb{Z})=-\chi_U(\pi'_{*}\mathbb{Z}).\]
\end{proof}
\subsection{Multiplicative Property}
\begin{defn}\label{defn_multiplicative}
Suppose we have a short exact sequence of strongly-$\mathbb{Z}$-constructible sheaves on $U$
\begin{equation}\label{chi_mul_seq1}
0 \to \mathcal{F}_1 \to \mathcal{F}_2 \to \mathcal{F}_3 \to 0.
\end{equation}
We say that  $\chi_U$ is multiplicative with respect to \eqref{chi_mul_seq1}
   if  $ {\chi}_U(\mathcal{F}_2) = {\chi}_U(\mathcal{F}_1){\chi}_U(\mathcal{F}_3).$
\end{defn}
Clearly if \eqref{chi_mul_seq1} splits then $\chi_U$ is multiplicative with respect to  \eqref{chi_mul_seq1}. 
We want to know whether $\chi_U$ is multiplicative or not in general. Unfortunately, we do not know the answer to this question except in special case when $K$ is totally imaginary and $U=X$ (see Proposition  \ref{prop_multi_imaginary}). Fortunately, we only need the multiplicative property of $\chi_U$ for a special type of exact sequences for the proof of our main results (see Proposition \ref{prop_weak_multiplicative}).

   From \eqref{eqn_thm_chi1}, ${{\chi}_U(\mathcal{F}_1){\chi}_U(\mathcal{F}_3)}/{{\chi}_U(\mathcal{F}_2)}$ is given by

\begin{equation}\label{chi_mul_eq0}
\prod_{i=1}^3\left(\frac{R(\mathcal{F}_i)}{[\mathrm{cok}(\delta(\mathcal{F}_i)_{tor})]}\right)^{(-1)^{i+1}}
  \prod_{i=1}^3\left(
  \frac{[H^0_{et}(U,\mathcal{F}_{i})_{tor}][Ext^1_{U}(\mathcal{F}_i,\mathbb{G}_m)]}{[Hom_{U}(\mathcal{F}_i,\mathbb{G}_m)_{tor}][H^0_{et}(U,\mathcal{F}_{i,B})_{tor}]}
  \right)^{(-1)^{i+1}} 
  \prod_{i=1}^3[\ker(\Psi^1(\mathcal{F}_i))]^{(-1)^{i}}.
\end{equation}

From the long exact sequence of cohomology associated with ($\ref{chi_mul_seq1}$) 
\begin{equation}\label{chi_mul_seq2}
(\mathcal{H}^0): \qquad 0 \to H^0_{et}(U,\mathcal{F}_1) \to H^0_{et}(U,\mathcal{F}_2) \to
 H^0_{et}(U,\mathcal{F}_3) \to Q_1 \to 0 .
\end{equation}
Applying Proposition $\ref{det_tor}$ to ($\ref{chi_mul_seq2}$), we obtain
\begin{equation}\label{chi_mul_eq2}
\left(\prod_{i=1}^3[H^0_{et}(U,\mathcal{F}_i)_{tor}]^{(-1)^{i+1}}\right)=
\nu(\mathcal{H}^0)_{\mathbb{R}}[Q_1].
\end{equation}
Similarly, we have 
\begin{equation}\label{chi_mul_seq4}
(\mathcal{H}om): \quad 0 \to Hom_U(\mathcal{F}_3,\mathbb{G}_m) \to Hom_U(\mathcal{F}_2,\mathbb{G}_m) \to Hom_U(\mathcal{F}_1,\mathbb{G}_m) \to R_1 \to 0.
\end{equation}
\begin{equation}\label{chi_mul_seq4b}
  0 \to R_1 \to Ext^1_U(\mathcal{F}_3,\mathbb{G}_m) \to Ext^1_U(\mathcal{F}_2,\mathbb{G}_m) \to Ext^1_U(\mathcal{F}_1,\mathbb{G}_m) \to R_2 \to 0.
\end{equation}
Applying Proposition $\ref{det_tor}$ to ($\ref{chi_mul_seq4}$) yields
\begin{equation}\label{chi_mul_eq3}
\left(\prod_{i=1}^3[Hom_{U}(\mathcal{F}_i,\mathbb{G}_m)_{tor}]^{(-1)^{i+1}}\right)
=\nu(\mathcal{H}om)_{\mathbb{R}}[R_1].
\end{equation}
From \eqref{chi_mul_seq4b}, we have
\begin{equation}\label{chi_mul_eq5}
\left(\prod_{i=1}^3[Ext^1_{U}(\mathcal{F}_i,\mathbb{G}_m)]^{(-1)^{i+1}}\right)
=[R_1][R_2].
\end{equation}
Let $(\mathcal{H}_S)$ be the exact sequence
\begin{equation}\label{chi_mul_seq5}
(\mathcal{H}_S): \quad 0 \to H^0_{et}(U,\mathcal{F}_{1,S}) \to H^0_{et}(U,\mathcal{F}_{2,S}) \to H^0_{et}(U,\mathcal{F}_{3,S}) \to P_1 \to 0.
\end{equation} 
Applying Proposition $\ref{det_tor}$ to ($\ref{chi_mul_seq5}$), we obtain
\begin{equation}\label{chi_mul_eq4}
 \left(\prod_{i=1}^3[H^0_{et}(U,\mathcal{F}_{i,S})_{tor}]^{(-1)^{i+1}}\right)=\nu(\mathcal{H}_S)_{\mathbb{R}}[P_1].
\end{equation}

Therefore,
\begin{equation}\label{eqn_multi_1}
\prod_{i=1}^3\left(
  \frac{[H^0_{et}(U,\mathcal{F}_{i})_{tor}][Ext^1_{U}(\mathcal{F}_i,\mathbb{G}_m)]}{[Hom_{U}(\mathcal{F}_i,\mathbb{G}_m)_{tor}][H^0_{et}(U,\mathcal{F}_{i,B})_{tor}]}
  \right)^{(-1)^{i+1}} 
  = \frac{\nu(\mathcal{H}^0)_{\mathbb{R}}[Q_1][R_2]}
  {\nu(\mathcal{H}om)_{\mathbb{R}}\nu(\mathcal{H}_S)_{\mathbb{R}}[P_1]}.
\end{equation}

\begin{prop}\label{prop_regulator_multi}
\begin{equation}\label{eqn_regulator_multi_1}
\prod_{i=1}^3\left(\frac{R(\mathcal{F}_i)}{[\mathrm{cok}(\delta(\mathcal{F}_i)_{tor})]}\right)^{(-1)^{i+1}}
=\frac{\nu(\mathcal{H}om)_{\mathbb{R}}\nu(\mathcal{H}_S)_{\mathbb{R}}}{\nu(\mathcal{H}^0)_{\mathbb{R}}}.
\end{equation}
\end{prop}
\begin{proof}
Consider the exact sequence
\begin{equation}\label{R(F)_T}
0 \to   H^0_{et}(U,\mathcal{F}_i)_{\mathbb{R}} \to
H_{et}^0(U,\mathcal{F}_{i,S})_{\mathbb{R}} 
\xrightarrow{\delta(\mathcal{F}_i)} \left(\frac{H^0_{et}(U,\mathcal{F}_{i,S})}{H^0_{et}(U,\mathcal{F}_i)}\right)_{\mathbb{R}}\to 0
\end{equation}
Since $\mathcal{F}_i$ is strongly-$\mathbb{Z}$-constructible, $\ker(\Psi^0(\mathcal{F}_i))$ is finite. Applying Lemma $\ref{torsion_group}$ to the exact sequence
\[ 0 \to \ker(\Psi^0(\mathcal{F}_i)) \to H^0_{et}(U,\mathcal{F}_i)\xrightarrow{\Psi^0(\mathcal{F}_i)} H^0_{et}(U,\mathcal{F}_{i,S}) \xrightarrow{\delta(\mathcal{F}_i)} 
\frac{H^0_{et}(U,\mathcal{F}_{i,S})}{H^0_{et}(U,\mathcal{F}_i)} \to 0 \] 
we deduce that ($\ref{R(F)_T}$) has determinant $[\mathrm{cok} (\delta(\mathcal{F}_i)_{tor})]$ 
with respect to integral bases.
Applying Lemma $\ref{det_2x3}$ to the following diagram
\[ \xymatrix{ 
0 \ar[r] & \left(\frac{H^0_{et}(U,\mathcal{F}_{1,S})}{H^0_{et}(U,\mathcal{F}_1)}\right)_{\mathbb{R}} 
 \ar[r] \ar[d] & 
\left(\frac{H^0_{et}(U,\mathcal{F}_{2,S})}{H^0_{et}(U,\mathcal{F}_2)}\right)_{\mathbb{R}}  \ar[r] \ar[d] & 
\left(\frac{H^0_{et}(U,\mathcal{F}_{3,S})}{H^0_{et}(U,\mathcal{F}_3)}\right)_{\mathbb{R}} \ar[d] \ar[r] & 0
& (\mathcal{H}_S/\mathcal{H}^0)  \\
0 \ar[r] & Hom_{U}(\mathcal{F}_1,\mathbb{G}_m)^{*}_{\mathbb{R}} \ar[r] & Hom_{U}(\mathcal{F}_2,\mathbb{G}_m)^{*}_{\mathbb{R}} \ar[r] & Hom_{U}(\mathcal{F}_3,\mathbb{G}_m)^{*}_{\mathbb{R}} \ar[r] & 0 
& (\mathcal{H}om)^{*}_{\mathbb{R}} \\}
\]
yields 
\begin{equation}\label{eqn_regulator_multi_2}
\left(\prod_{i=1}^3R(\mathcal{F}_i)^{(-1)^{i+1}}\right)
={\nu(\mathcal{H}_S/\mathcal{H}^0)\nu(\mathcal{H}om)_{\mathbb{R}}}  .
\end{equation}
 Consider  the following diagram 
\[\xymatrixrowsep{0.2in}\xymatrix{    
&  0  \ar[d] &    0 \ar[d]  &   0  \ar[d]  \\
0 \ar[r] & H^0_{et}(U,\mathcal{F}_1)_{\mathbb{R}}  \ar[d] \ar[r]   &  H^0_{et}(U,\mathcal{F}_2)_{\mathbb{R}} \ar[d] 
\ar[r] &H^0_{et}(U,\mathcal{F}_3)_{\mathbb{R}}       \ar[d] \ar[r] & 0 & (\mathcal{H}^0)_{\mathbb{R}} 
\\
0 \ar[r] & H^0_{et}(U,\mathcal{F}_{1,S})_{\mathbb{R}}  \ar[d]^{\delta(\mathcal{F}_1)} \ar[r] & H^0_{et}(U,\mathcal{F}_{2,S})_{\mathbb{R}} \ar[d]^{\delta(\mathcal{F}_2)} \ar[r] & H^0_{et}(U,\mathcal{F}_{3,S})_{\mathbb{R}} \ar[d]^{\delta(\mathcal{F}_3)}\ar[r] & 0 & (\mathcal{H}_S)_{\mathbb{R}}
\\
0 \ar[r] & \left(\frac{H^0_{et}(U,\mathcal{F}_{1,S})}{H^0_{et}(U,\mathcal{F}_1)}\right)_{\mathbb{R}}  \ar[r] \ar[d] & 
\left(\frac{H^0_{et}(U,\mathcal{F}_{2,S})}{H^0_{et}(U,\mathcal{F}_2)}\right)_{\mathbb{R}}  \ar[r]  \ar[d] & 
\left(\frac{H^0_{et}(U,\mathcal{F}_{3,S})}{H^0_{et}(U,\mathcal{F}_3)}\right)_{\mathbb{R}} \ar[r] \ar[d] & 0 &
 (\mathcal{H}_S/\mathcal{H}^0)  \\ 
 &  0   &    0   &   0   \\                    
  & (\mathcal{E}_1) & (\mathcal{E}_2) & (\mathcal{E}_3)   &             \\   }
 \]
 Applying Lemma $\ref{det_3x3}$ and note that $\nu(\mathcal{E}_i)=[\mathrm{cok} (\delta(\mathcal{F}_i)_{tor})]$, we have  
 \begin{equation}\label{eqn_regulator_multi_3}
 \left(\prod_{i=1}^3[\mathrm{cok}(\delta(\mathcal{F}_i)_{tor})]^{(-1)^{i+1}}\right) =
 \frac{\nu(\mathcal{H}^0)_{\mathbb{R}}\nu(\mathcal{H}_S/\mathcal{H}^0)}{\nu(\mathcal{H}_S)_{\mathbb{R}}}.
 \end{equation}
 Therefore, \eqref{eqn_regulator_multi_1} follows from \eqref{eqn_regulator_multi_2} and \eqref{eqn_regulator_multi_3}.
\end{proof}

Combining \eqref{eqn_multi_1} and \eqref{eqn_regulator_multi_1} yields
\begin{equation}\label{eqn_multi_2}
\frac{{\chi}_U(\mathcal{F}_1){\chi}_U(\mathcal{F}_3)}{{\chi}_U(\mathcal{F}_2)}=
\frac{[Q_1][R_2]}{[P_1]}\prod_{i=1}^3[\ker(\Psi^1(\mathcal{F}_i))]^{(-1)^{i}}
\end{equation}
\begin{prop}\label{prop_multi_imaginary}
Suppose $K$ is a totally imaginary number field and $U=X=Spec(O_K)$. 
 Then $\chi_U=\chi$ is multiplicative with respect to every short exact sequence of strongly-$\mathbb{Z}$-constructible sheaves on $U$. 
\end{prop}
\begin{proof}
When $U=X$ and $K$ is totally imaginary, $P_1=0$ and $\ker\Psi^1(\mathcal{F}_i)=H^1_{et}(X,\mathcal{F}_i)$. The Artin-Verdier duality in this special case implies 
$Ext^1_X(\mathcal{F}_i,\mathbb{G}_m)\simeq H^2_{et}(X,\mathcal{F}_i)^D$. Therefore, we have the following exact sequence
\[ 0 \to Q_1 \to  H^1_{et}(X,\mathcal{F}_1) \to H^1_{et}(X,\mathcal{F}_2) \to H^1_{et}(X,\mathcal{F}_3) \to R_2^D \to 0.\]
In particular,
\[ \prod_{i=1}^3[H^1_{et}(X,\mathcal{F}_i)]^{(-1)^{i}} = ([Q_1][R_2])^{-1}.
\]
 Then \eqref{eqn_multi_2} implies $\chi(\mathcal{F}_2)=\chi(\mathcal{F}_1)\chi(\mathcal{F}_3)$.
\end{proof}

\begin{prop}\label{prop_weak_multiplicative}
Suppose we have an exact sequence of strongly-$\mathbb{Z}$-constructible sheaves on $U$
\begin{equation*}
0 \to \mathcal{F}_1 \to \mathcal{F}_2 \to \mathcal{F}_3 \to 0.
\end{equation*}
Assume further that  $H^1_{et}(U,\mathcal{F}_2)=0$ and $H^1(K_v,\mathcal{F}_{2,v})=0$ for all $v\in S$.
   Then \[ {\chi}_U(\mathcal{F}_2) = {\chi}_U(\mathcal{F}_1){\chi}_U(\mathcal{F}_3).\]
\end{prop}
\begin{proof}
We use the same notations used in the beginning of the section. As $H^1_{et}(U,\mathcal{F}_2)=0$ and $H^1(K_v,\mathcal{F}_{2,v})=0$ for $v\in S$, we have $Q_1=H^1_{et}(U,\mathcal{F}_1)$, $P_1=\prod_{v\in S} H^1(K_v,\mathcal{F}_{2,v})$, $\ker(\Psi^1(\mathcal{F}_1))=0$. Therefore, from \eqref{eqn_multi_2} 
\[ \frac{{\chi}_U(\mathcal{F}_1){\chi}_U(\mathcal{F}_3)}{{\chi}_U(\mathcal{F}_2)}
= \frac{[R_2]}{[\mathrm{cok}(\Psi^1(\mathcal{F}_1))][\ker(\Psi^1(\mathcal{F}_3))]}.\]
Clearly, $\ker(\Psi^2(\mathcal{F}_1))\simeq H^2_c(U,\mathcal{F}_1)/\mathrm{cok}(\Psi^1(\mathcal{F}_1))$. We consider the composition of maps
\[ H^2_c(U,\mathcal{F}_1)
 \to H^2_c(U,\mathcal{F}_1)/\mathrm{cok}(\Psi^1(\mathcal{F}_1)) \xrightarrow{\beta} H^2_c(U,\mathcal{F}_2). \]
 Then by the kernel-cokernel exact sequence, we deduce
 \[ 0 \to \mathrm{cok}(\Psi^1(\mathcal{F}_1)) \to R_2^D \to \ker(\beta) \to 0.\]
 Therefore, it remains to prove $\ker(\beta)\simeq \ker(\Psi^1(\mathcal{F}_3))$. Indeed, consider the following diagram
 \[  \xymatrix{
 &  &  0  \ar[d] &  0  \ar[d]   \\
 &  & H^1_{et}(U,\mathcal{F}_3) \ar[d]  \ar[r]^{\Psi^1(\mathcal{F}_3)} &   \prod_{v\in S} H^1(K_v,\mathcal{F}_{3,v})  \ar[d]  \\
 0  \ar[r]  &   \ker(\Psi^2(\mathcal{F}_1))  \ar[r] \ar[d]^{\beta} &     H^2_{et}(U,\mathcal{F}_1) \ar[r]  \ar[d]  &   \prod_{v\in S} H^2(K_v,\mathcal{F}_{1,v})  \ar[d] \\
 0  \ar[r]  &    H^2_c(U,\mathcal{F}_2)   \ar[r]  &     H^2_{et}(U,\mathcal{F}_2) \ar[r]    &   \prod_{v\in S} H^2(K_v,\mathcal{F}_{2,v})  
 } 
 \]
 By diagram chasing, $\ker(\beta)\simeq \ker(\Psi^1(\mathcal{F}_3))$. Hence, $ {\chi}_U(\mathcal{F}_2) = {\chi}_U(\mathcal{F}_1){\chi}_U(\mathcal{F}_3).$
\end{proof}
\section{Artin L-functions of toric type}
\subsection{Special Values of L-functions}
Let $K$ be a number field. Let $M$ be a torsion free discrete $G_K$-module of finite type. Let $L/K$ be a finite Galois extension such that $G_L$ acts trivially on $M$, in other words, $L$ is a splitting field of $M$. Then $M_{\mathbb{C}}:=M\otimes_{\mathbb{Z}}\mathbb{C}$ is a finite dimensional representation of $G=G(L/K)$. 

For each finite place $v$ of $K$, let $w$ be a place of $L$ lying above $v$. Let $D_w$, $I_w$ be the decomposition and inertia groups of $w$. Let $F_w$ be the Frobenius element at $w$ and
$f_w$ be the inertia degree. Then $D_w/I_w$ is a cyclic group of order $f_w$ generated by $F_w$. We write $N(v)$ for the norm of $v$.
We recall the definition of the Artin L-function associated to $M_{\mathbb{C}}$ below.
\begin{defn}\label{defn_local_L}
 The local  L-function is defined by
\[ L_v(M,s):=\det(I-N(v)^{-s}F_w|M_{\mathbb{C}}^{I_w})^{-1}.\]
\end{defn}
\begin{defn}\label{defn_L_func}
Let $S$ be a finite set of places of $K$ containing all the infinite places. The partial Artin L-function is defined by
\[ L_S(M,s):=\prod_{v\notin S}L_v(M,s) \quad
\mbox{for $\mathrm{Re}(s)>1$}. \] 
We also write $L(M,s)$ for $L_{S_{\infty}}(M,s)$. 
\end{defn}
Then $L_S(M,s)$ is holomorphic for $\mathrm{Re}(s)>1$ and has a meromorphic continuation to the complex plane. Let $r_S(M):=\mathrm{ord}_{s=0}L_S(M,s)$ and $L_S^{*}(M,0):=\lim_{s\to 0}{L_S(M,s)}{s^{-r_S(M)}}$. We want to give a formula for $r_S(M)$ and $L_S^{*}(M,0)$ in terms of the Weil-\'etale Euler characteristic constructed in the previous sections. We begin by giving a cohomological formula for the order of vanishing and special value at zero of $L_v(M,s)$.
\begin{prop}\label{prop_local_L}
Let $v$ be a finite prime of $K$ and $w$ be a prime of $L$ lying over $v$.
\begin{enumerate}
\item Let $r_v(M)=\mathrm{rank}_{\mathbb{Z}}H^0(K_v,M)$. Then $ord_{s=0}L_v(M,s)=-r_v(M)$.
\item Let  $L_v^{*}(M,0):=\lim_{s\to 0}{L_v(M,s)}{s^{r_v(M)}}$ and $h(D_w/I_w,M^{I_w})$ be the Herbrand quotient. Then 
\[ L_v^{*}(M,0)=\frac{h(D_w/I_w,M^{I_w})}{(f_w\log N(v))^{r_v(M)}}.\]
\end{enumerate}
\end{prop}
\begin{proof}
To ease notations, let $H:=D_w/I_w$.  Let $V=M^{I_w}_{\mathbb{C}}$ and $\pi_H : V \to V^H$ be the projection map
\[ \pi_H(x):=\frac{1}{f_w}\sum_{n=0}^{f_w-1}F_w^n(x).\]
Let $W=\ker(\pi_H)$. We have the following exact sequence of $H$-spaces.
\begin{equation}
0 \to W \to V \xrightarrow{\pi_H}  V^{H} \to 0
\end{equation} 
In particular, $L_v(M,s)^{-1}=\det(I-N(v)^{-s}F_w|V^H)\det(I-N(v)^{-s}F_w|W)$. As $F_w$ acts trivially on $V^H$ and $V^H\simeq H^0(K_v,M)_{\mathbb{C}}$, 
\[ \det(I-N(v)^{-s}F_w|V^H)=(1-N(v)^{-s})^{r_v(M)}.\]
On the other hand, $\lim_{s\to 0}\det(I-N(v)^{-s}F_w|W)=\det(I-F_w|W)\neq 0$. Indeed, suppose
$\det(I-F_w|W)= 0$, then there exists a non-zero element $x$ in $W$ such that $F_w(x)=x$. Hence, $0=\pi_H(x)=x$ which is a contradiction.  Therefore, 
\[ \mathrm{ord}_{s=0}L_v(M,s)=\mathrm{ord}_{s=0}(1-N(v)^{-s})^{-r_v(M)}=-r_v(M).\]
It is not hard to see that 
\[ \lim_{s\to 0} \frac{1-N(v)^{-s}}{s} = \log N(v).\]
Therefore, it remains to compute $\det(I-F_w|W)$. Let $N_H:=f_w\pi_H$, i.e. $N_H$ is the usual norm map in group cohomology. Let $W':=\ker(N_H|M^{I_w})$. Then 
$W\simeq W'\otimes_{\mathbb{Z}}\mathbb{C}$ and  $\det(I-F_w|W)= \det(I-F_w|W')$. We have the following exact sequence 
\begin{equation}
0 \to W' \xrightarrow{I-F_w} W' \to H^{-1}_T(H,W') \to 0.
\end{equation}
Note that $I-F_w$ is injective on $W'$ because $M$ is torsion free. Therefore, $\det(I-F_w|W')=[ H^{-1}_T(H,W')] =  [ H^{1}(H,W')]$. 
Consider the following exact sequence of $H$-modules
\[ 0 \to W' \to M^{I_w} \xrightarrow{N_H} N_H(M^{I_w}) \to 0.\]
The exact sequence of cohomology yields
\[ 0 \to  H^{0}_T(H,M^{I_w}) \xrightarrow{N_H} H^{0}_T(H,N_H(M^{I_w})) \to H^{1}(H,W')
\to  H^{1}(H,M^{I_w}) \to 0.
\]
Since the rank of $N_H(M^{I_w})$ is $r_v(M)$ and $[H]=f_w$, we have 
\[ [ H^{1}(H,W')]=\frac{f_w^{r_v(M)}[H^{1}(H,M^{I_w})]}{[H^{0}_T(H,M^{I_w})]}
=\frac{f_w^{r_v(M)}}{h(H,M^{I_w})}.\]
As a result,
\[ L_v^{*}(M,0)=\frac{\lim_{s\to 0}\left(\frac{1-N(v)^{-s}}{s}\right)^{-r_v(M)}}{\det(I-F_w|W)}
=\frac{h(D_w/I_w,M^{I_w})}{(f_w\log N(v))^{r_v(M)}}.\]
\end{proof}

\begin{thm}\label{euler_value_L0}
Let $M$ be a torsion free discrete $G_K$-module of finite type. Then 
\begin{enumerate}
	\item $\mathrm{ord}_{s=0}L_S(M,s)= E_U(j_{*}M)=
	\mathrm{rank}_{\mathbb{Z}} Hom_{U}(j_{*}M,\mathbb{G}_m) $.
	\item $ L^{*}_S(M,0)=\pm\chi_U(j_{*}M) $.
\end{enumerate}
\end{thm}
\begin{proof}
\begin{enumerate}
\item 
From Theorem \ref{thm_chi1_formula}, $E_U(j_{*}M)=
	\mathrm{rank}_{\mathbb{Z}} Hom_{U}(j_{*}M,\mathbb{G}_m) $. 
The ranks of
$Hom_{U}(j_{*}M,\mathbb{G}_m)$ and $\prod_{v \in S} H^0(K_v,M_v)/H^0(K,M)$ are the same because the regulator pairing for $j_{*}M$ is non-degenerate.
Thus, from $\cite[\mbox{I.3.4}]{Tat84}$,
\begin{eqnarray*}
\mathrm{ord}_{s=0}L_S(M,s)&=& \sum_{v \in S} \mathrm{rank}_{\mathbb{Z}} H^0(K_v,M_v) - \mathrm{rank}_{\mathbb{Z}}H^0(K,M)\\
&=& \mathrm{rank}_{\mathbb{Z}} Hom_{U}(j_{*}M,\mathbb{G}_m)= E_U(j_{*}M).
\end{eqnarray*} 
\item
Consider the two exact sequences ($\ref{ono_etale_seq1}$) and ($\ref{ono_etale_seq2}$) from Proposition $\ref{ono_etale}$. Since $\mathcal{Q}$ is a constructible sheaf, $\chi(\mathcal{Q})=1$. By Propositions $\ref{euler_piZ}$
, $\chi_U((\pi_{\lambda}')_{*}\mathbb{Z})=-L^{*}_S((\pi_{\lambda})_{*}\mathbb{Z},0)$ and $\chi_U((\pi_{\mu}')_{*}\mathbb{Z})=-L_S^{*}((\pi_{\mu})_{*}\mathbb{Z},0)$. 
Hence, by Proposition  $\ref{prop_weak_multiplicative}$ and the fact that $N$ is finite
\begin{eqnarray*}
\chi_U(j_{*}M)^n &=& \frac{\prod_{\lambda}\chi_U((\pi_{\lambda}')_{*}\mathbb{Z})}
{\prod_{\mu}\chi_U((\pi_{\mu}')_{*}\mathbb{Z})}  
=  \left|\frac{\prod_{\lambda}L_S^{*}((\pi_{\lambda})_{*}\mathbb{Z},0)}
{\prod_{\mu}L_S^{*}((\pi_{\mu})_{*}\mathbb{Z},0)}\right| 
=  |L_S^{*}(M,0)^n|.
\end{eqnarray*}
Since $L_S^{*}(M,0)$ is a real number, we deduce $L_S^{*}(M,0)=\pm\chi_U(j_{*}M)$.
\end{enumerate}
\end{proof}
\begin{rmk}\label{rmk_sign}
The ambiguity of the sign in equation $L_S^{*}(M,0)=\pm\chi_U(j_{*}M)$ is because we do not know the exact form of sequence ($\ref{ono_etale_seq1}$).  If we know the form of ($\ref{ono_etale_seq1}$) and $n=1$, then we can determine the sign exactly. For example, if $M$ is the character group of a norm torus, then $L_S^{*}(M,0)=\chi_U(j_{*}M)$.
\end{rmk}
\begin{cor}\label{cor_L_ST}
Let $T$ be an algebraic torus over a number field $K$ with character group $\hat{T}$. Let $\Psi^1(j_{*}\hat{T})$ be the map
$H^1_{et}(U,j_{*}\hat{T}) \to \prod_{v\in S}H^1(K_v,\hat{T})$. Then 
$\mathrm{ord}_{s=0}L_S(\hat{T},s)=\mathrm{rank}_{\mathbb{Z}}T(O_{K,S})$ and 
\begin{equation}\label{eqn_cor_LST}
L^{*}_S(\hat{T},0)=\pm \frac{[Ext^1_U(j_{*}\hat{T},\mathbb{G}_m)]R_{T,S}}{[\ker\Psi^1(j_{*}\hat{T})]w_T}.
\end{equation}
\end{cor}
\begin{proof}
From Theorem $\ref{euler_value_L0}$, \eqref{eqn_thm_chi1} and the fact that $\hat{T}$ is torsion free, we have
\[ L_S^{*}(\hat{T},0)
=\pm\frac{[Ext^1_U(j_{*}\hat{T},\mathbb{G}_m)]R(j_{*}\hat{T})}{[\ker\Psi^1(j_{*}\hat{T})][Hom_U(j_{*}\hat{T},\mathbb{G}_m)_{tor}]}.
\]
From $\cite[\mbox{2.3 \& 4.4}]{Tra16}$, $T(O_{K,S})\simeq Hom_U(j_{*}\hat{T},\mathbb{G}_m)$ and $T(O_K)_{tor}\simeq T(O_{K,S})_{tor}$. Moreover, $R(j_{*}\hat{T})\simeq R_{T,S}$. Hence, the corollary follows. 
\end{proof}
\begin{thm}\label{L_ST}
Let $K$ be a number field and $T$ be an algebraic torus over $K$ with character group $\hat{T}$.
Let $h_{T,S}$, $R_{T,S}$ and $w_T$ be the $S$-class number, the $S$-regulator and the number of roots of unity of $T$.
Let $\mathbb{III}^1(T)$ be the Tate-Shafarevich group. Then
\begin{eqnarray}\label{eqn_LST}
L^{*}_S(\hat{T},0)=\pm \frac{h_{T,S}R_{T,S}}{w_T}\frac{[\mathbb{III}^1(T)]}{[H^{1}(K,\hat{T})]}\prod_{v\in S}{[H^1(K_v,\hat{T})]} \prod_{v \notin S}[H^0(\hat{\mathbb{Z}},H^1(I_v,\hat{T}))]
.
\end{eqnarray}
\end{thm}
\begin{proof}
From $\cite{Tra16}$, we have the following formula for the $S$-class number of $T$
\begin{eqnarray}\label{eqn_hTS}
 h_{T,S} = 
\frac{[Ext^1_U(j_{*}\hat{T},\mathbb{G}_m)][H^1(K,\hat{T})] }
{[\mathrm{ker}\Psi^1(j_{*}\hat{T})][\mathbb{III}^1(T)]
\prod_{v\in S}[H^1(K_v,{T})]
\prod_{v\notin S}[H^0(\hat{\mathbb{Z}},H^1(I_v,\hat{T}))]
}.
\end{eqnarray}
Thus, the theorem follows from the above formula, Corollary \ref{cor_L_ST} and the fact that
$[H^1(K_v,{T})]=[H^1(K_v,\hat{T})]$.
\end{proof}

Formula \eqref{eqn_LST} yields the following formula relating $h_TR_T$ and $h_{T,S}R_{T,S}$. 

\begin{prop}\label{prop_hR_hRS}
Let $T$ be an algebraic torus defined over a number field $K$. Let $S$ be a finite set of places of $K$ containing $S_{\infty}$. Suppose $L/K$ is a Galois splitting field  of $T$.
For each finite place $v$ of $K$, let $w$ be a place of $L$ dividing $v$. Let $D_w$, $I_w$ be the decomposition and inertia groups of $w$ and let $f_w$ be the inertia degree. 
 We write $r_v(\hat{T})$ for $\mathrm{rank}_{\mathbb{Z}}H^0(K_v,\hat{T})$. Then 
\begin{equation}\label{prop_hTS_eqn}
\frac{h_{T,S}R_{T,S}}{h_TR_T} = \prod_{v\in S-S_{\infty}} \frac{[H^0(\hat{\mathbb{Z}},H^1(I_v,\hat{T}))]}{[H^1(K_v,\hat{T})]} 
\frac{(f_w\log N(v))^{r_v(\hat{T})}}{h(D_w/I_w,\hat{T}^{I_w})}.
\end{equation}
\end{prop}  
\begin{proof}
From \eqref{eqn_LST} and the fact that $L_S^{*}(\hat{T},0)/L^{*}(\hat{T},0)=\prod_{v\in S-S_{\infty}}L_v^{*}(\hat{T},0)^{-1}$, we have
\begin{eqnarray}
\frac{h_{T,S}R_{T,S}}{h_TR_T} =  \prod_{v\in S-S_{\infty}} \frac{[H^0(\hat{\mathbb{Z}},H^1(I_v,\hat{T}))]}{[H^1(K_v,\hat{T})]L_v^{*}(\hat{T},0)}. 
\end{eqnarray}
Thus, \eqref{prop_hTS_eqn} follows from Proposition \ref{prop_local_L}.
\end{proof}
\begin{examp}
If $T=\mathbb{G}_m$ then \eqref{prop_hTS_eqn} becomes 
$ h_{K,S}R_{K,S}= h_KR_K \prod_{v\in S-S_{\infty}}  \log N(v).$
\end{examp}

\subsection{The functional equation}
Using the functional equation of Artin $L$-functions, we can obtain a formula for $L^{*}(\hat{T},1)$ up to signs.
Let $f(\hat{T})$ be the Artin conductor of $\hat{T}$ and $d=\mathrm{rank}_{\mathbb{Z}}\hat{T}$. We write $A:=N_{K/\mathbb{Q}}(f(\hat{T}))|\Delta_K|^d$. For each infinite place $v$ of $K$ we define $r_v(\hat{T}):=\mathrm{rank}_{\mathbb{Z}}H^0(K_v,\hat{T})$ and
\[ L_v(\hat{T},s):=\left\{ \begin{array}{ll}
\Gamma_{\mathbb{C}}(s)^d & \mbox{$v$ complex} \\
\Gamma_{\mathbb{R}}(s)^{r_v(\hat{T})}\Gamma_{\mathbb{R}}(s+1)^{d-r_v(\hat{T})} &
\mbox{$v$ real.}
\end{array}
\right. \]
The complete Artin $L$-function is defined as 
\[ \Lambda(\hat{T},s):=A^{s/2}\prod_{v\in S_{\infty}}L_v(\hat{T},s)L(\hat{T},s).\]
Then from \cite[\mbox{page 18}]{Tat84}, there exists a constant $w(\hat{T})=\pm 1$ such that 
\[ \Lambda(\hat{T},s)=w(\hat{T})\Lambda(\hat{T},1-s).\]
\begin{thm}\label{thm_LST_1}
Let $K$ be a number field and $T$ be an algebraic torus of dimension $d$ over $K$ with character group $\hat{T}$.
Let $h_{T}$, $R_{T}$ and $w_T$ be the class number, the regulator and the number of roots of unity of $T$.
Let $\mathbb{III}^1(T)$ be the Tate-Shafarevich group and let 
\[ \Omega_{\infty}(\hat{T}):= \prod_{\mbox{$v$ real}}2^{r_v(\hat{T})}\pi^{d-r_v(\hat{T})}[H^1(K_v,\hat{T})] \prod_{\mbox{$v$ complex}} (2\pi)^d.\]
Then $\mathrm{ord}_{s=1}L(\hat{T},s)=-\mathrm{rank}_{\mathbb{Z}}H^0(K,\hat{T})$ and 
\begin{eqnarray}\label{eqn_LST_1}
L^{*}(\hat{T},1)=\pm \frac{h_{T}R_{T}}{w_T}\frac{[\mathbb{III}^1(T)]}{[H^{1}(K,\hat{T})]} 
\frac{\Omega_{\infty}(\hat{T})}{N_{K/\mathbb{Q}}(f(\hat{T}))^{1/2}|\Delta_K|^{d/2}}
\prod_{v \notin S_{\infty}}[H^0(\hat{\mathbb{Z}},H^1(I_v,\hat{T}))].
\end{eqnarray}
\end{thm}

\begin{proof}
Let $r_K(\hat{T}):=\mathrm{rank}_{\mathbb{Z}}H^0(K,\hat{T})$. Then $\lim_{s\to 0}\Lambda(\hat{T},s)s^{r_K(\hat{T})}$ is given by
\begin{eqnarray*}
&=& 
\lim_{s\to 0} A^{s/2}\prod_{\mbox{$v$ real}} (\Gamma_{\mathbb{R}}(s)s)^{r_v(\hat{T})}
(\Gamma_{\mathbb{R}}(s+1))^{d-r_v(\hat{T})} 
\prod_{\mbox{$v$ complex}} (\Gamma_{\mathbb{C}}(s)s)^{d}
\frac{L(\hat{T},s)}{s^{\sum_{v\in S_{\infty}}r_v(\hat{T})-r_K(\hat{T})}} \\
&=& \left(\prod_{\mbox{$v$ real}}2^{r_v(\hat{T})}\prod_{\mbox{$v$ complex}} 2^d\right) L^{*}(\hat{T},0).
\end{eqnarray*}
On the other hand, $\lim_{s\to 0}\Lambda(\hat{T},1-s)s^{r_K(\hat{T})}$ is given by
\begin{eqnarray*}
&=& 
\lim_{s\to 0} A^{(1-s)/2}\prod_{\mbox{$v$ real}} \Gamma_{\mathbb{R}}(1-s)^{r_v(\hat{T})}
\Gamma_{\mathbb{R}}(2-s)^{d-r_v(\hat{T})} 
\prod_{\mbox{$v$ complex}} \Gamma_{\mathbb{C}}(1-s)^{d}
{L(\hat{T},1-s)}{s^{r_K(\hat{T})}} \\
&=& A^{1/2}\left(\prod_{\mbox{$v$ real}}\pi^{r_v(\hat{T})-d}\prod_{\mbox{$v$ complex}} \pi^{-d}\right) (-1)^{r_K(\hat{T})}L^{*}(\hat{T},1).
\end{eqnarray*}
As $\Lambda(\hat{T},1-s)=\pm \Lambda(\hat{T},s)$, we deduce
\begin{equation}\label{L1_L0}
\frac{L^{*}(\hat{T},1)}{L^{*}(\hat{T},0)}=\pm \frac{\prod_{\mbox{$v$ real}}2^{r_v(\hat{T})}\pi^{d-r_v(\hat{T})}\prod_{\mbox{$v$ complex}} (2\pi)^d}
{N_{K/\mathbb{Q}}(f(\hat{T}))^{1/2}|\Delta_K|^{d/2}}.
\end{equation}
The theorem then follows from \eqref{eqn_LST}. 
\end{proof}
\subsection{Examples : Norm Tori of Quadratic fields}
Let $d$ be a square-free integer. Let $K=\mathbb{Q}(\sqrt{d})$ and $T=R_{K/\mathbb{Q}}^{(1)}(\mathbb{G}_m)$ be the norm torus. We want to illustrate the results of this section using $T$. 
 Let $\pi : Spec(K)\to Spec(\mathbb{Q})$ and $\pi' : Spec(O_K)\to Spec(\mathbb{Z})=X$.  Then $\hat{T}$ satisfies the exact sequence
\[ 0 \to \mathbb{Z} \to \pi_{*}\mathbb{Z} \to \hat{T} \to 0.\]
Let $j: Spec(\mathbb{Q})\to X$. As $R^1j_{*}\mathbb{Z}=0$, we have
\[ 0 \to \mathbb{Z} \to \pi'_{*}\mathbb{Z} \to j_{*}\hat{T} \to 0\]
which yields the exact sequence

\[
0 \to Hom_{X}(j_{*}\hat{T},\mathbb{G}_m) \to O_K^{*} \xrightarrow{N_{K/\mathbb{Q}}} \{\pm 1\} \to 
Ext^1_{X}(j_{*}\hat{T},\mathbb{G}_m) \to Pic(O_K) \to 0.  \]
Note that $\ker(\Psi^1(j_{*}\hat{T}))=0$ by \cite[\mbox{Remark 5.3}]{Tra16}. Thus,
\[ L^{*}(\hat{T},0)= \frac{[Ext^1_X(j_{*}\hat{T},\mathbb{G}_m)]R_{T}}{w_T}.\]
\begin{enumerate}
\item 
Suppose $d<0$ i.e. $K$ is an imaginary quadratic field. We have $O_K^{*}=\mu_K$ which is finite. Hence, $Hom_{X}(j_{*}\hat{T},\mathbb{G}_m)$ is finite and $R_T=1$. Moreover, $N_{K/\mathbb{Q}}(\mu_K)=\{1\}$. Therefore, 
$Hom_{X}(j_{*}\hat{T},\mathbb{G}_m)\simeq \mu_K$ and 
$[Ext^1_{X}(j_{*}\hat{T},\mathbb{G}_m)]=2h_K$. As a result,
\begin{equation}
L^{*}(\hat{T},0)=\frac{2h_K}{w_K}.
\end{equation}
Let $v$ be the only infinite prime of $\mathbb{Q}$. Then $H^0(K_v,\hat{T})=0$ hence $r_v(\hat{T})=0$. Moreover, by the conductor-discriminant formula $|f(\hat{T})|=|\Delta_{\mathbb{Q}(\sqrt{-d})/\mathbb{Q}}|$. Therefore, by \eqref{L1_L0} 
\begin{equation}
L^{*}(\hat{T},1)=\pm \frac{2\pi h_K}{\sqrt{|\Delta_{\mathbb{Q}(\sqrt{-d})/\mathbb{Q}}|}w_K}.
\end{equation}
\item Suppose $d>0$ i.e. $K$ is a real quadratic field. Let $\epsilon$ be the fundamental unit of $K$. Then $O_K^{*}\simeq \{\pm 1\}\times \epsilon^{\mathbb{Z}}$.
\begin{itemize}
\item If $N_{K/\mathbb{Q}}(\epsilon)=1$ then $Hom_{X}(j_{*}\hat{T},\mathbb{G}_m)\simeq O_K^{*}$. Hence, $R_T=\log(\epsilon)$ and $w_T=2$. Moreover, $[Ext^1_{X}(j_{*}\hat{T},\mathbb{G}_m)]=2h_K$. 
\item If $N_{K/\mathbb{Q}}(\epsilon)=-1$ then $Hom_{X}(j_{*}\hat{T},\mathbb{G}_m)\simeq \{\pm 1\}\times \epsilon^{2\mathbb{Z}}$. Hence, $R_T=2\log(\epsilon)$ and $w_T=2$. Moreover, $[Ext^1_{X}(j_{*}\hat{T},\mathbb{G}_m)]=h_K$. 
\end{itemize}
Either way, we still obtain 
\begin{equation}
L^{*}(\hat{T},0)=h_K\log(\epsilon).
\end{equation}
Let $v$ be the only infinite prime of $\mathbb{Q}$. Then $H^0(K_v,\hat{T})\simeq \mathbb{Z}$ hence $r_v(\hat{T})=1$. Moreover, by the conductor-discriminant formula $|f(\hat{T})|=|\Delta_{\mathbb{Q}(\sqrt{d})/\mathbb{Q}}|$. Therefore, by \eqref{L1_L0}
\begin{equation}
L^{*}(\hat{T},1)=\pm \frac{2h_K\log(\epsilon)}{\sqrt{|\Delta_{\mathbb{Q}(\sqrt{d})/\mathbb{Q}}|}}.
\end{equation}
\end{enumerate}

\section{Appendix: Determinants and Torsions}
We review some results about determinants of exact sequences and orders of torsion subgroups of finitely generated abelian groups.
\subsection{Determinants of Exact Sequences}
For $n\geq 1$, consider the following exact sequence of vector spaces over $\mathbb{R}$
\begin{equation*}
0 \to V_0 \xrightarrow{T_0} V_1 \xrightarrow{T_1} ... \xrightarrow{T_{n-1}} V_n \to 0  \quad \quad  (\mathcal{E}).
\end{equation*}
Let $B_i$ be an ordered basis for $V_i$. We want to define the determinant $\nu(\mathcal{E})$ of $(\mathcal{E})$ with respect to the bases $\{B_i\}$. We shall do so inductively.
\begin{enumerate}
\item If $n=1$, then  $\nu(\mathcal{E}):=|\det(T_0)|$ with respect to the given bases.
\item If $n=2$, suppose $B_0=\{u_i\}_{i=1}^{r}$, $B_1=\{v_i\}_{i=1}^{r+s}$ and $B_2=\{w_i\}_{i=1}^{s}$.
For $i=1,...,s$, let ${T_1}^{-1}(w_i)$ be any preimage of $w_i$ under $T_2$. We can form the following elements 
$\wedge_{i=1}^{r+s} v_i$ and $(\wedge_{i=1}^{r} T_0(u_i)) \wedge (\wedge_{i=1}^{s} T_1^{-1}(w_i))$ of $\wedge_{i=1}^{r+s} V_1$. Since $\wedge_{i=1}^{r+s} V_1$ is a 1 dimensional vector space over $\mathbb{R}$, there exists a unique positive real number $\delta$ such that 
\[  (\wedge_{i=1}^{r} T_0(u_i)) \wedge (\wedge_{i=1}^{s} T_1^{-1}(w_i)) = \pm \delta (\wedge_{i=1}^{r+s} v_i). \]
Note that the choice of the preimages of $w_i$  under $T_1$ does not affect 
$(\wedge_{i=1}^{r} T_0(u_i)) \wedge (\wedge_{i=1}^{s} T_1^{-1}(w_i))$.
Therefore we can define $\nu(\mathcal{E}):=\delta$.
\item If $n\geq 3$, suppose we have defined $\nu(\mathcal{E})$ for $n=N$. We want to define $\nu(\mathcal{E})$ for $n=N+1$.
Let $I$ be the image of $T_{N-1}$ and choose any basis for $I$. We split $(\mathcal{E})$ into 
\[  0 \to V_0 \xrightarrow{T_0} V_1 \xrightarrow{T_1} ... \xrightarrow{T_{N-2}} V_{N-1} \xrightarrow{T_{N-1}} I \to 0  \quad \quad  (\mathcal{E}_1),
\]
\[ 0 \to I \to V_{N} \xrightarrow{T_N} V_{N+1} \to 0 \quad \quad  (\mathcal{E}_2). \]
The determinant of ($\mathcal{E}$) defined to be
$ \nu(\mathcal{E}):=\nu(\mathcal{E}_1)\nu(\mathcal{E}_2)^{(-1)^{N-1}} $. 
Note that $\nu(\mathcal{E})$ is independent of the choice of basis for $I$.
\end{enumerate}

\begin{rmk}\label{rmk_determinant}
\begin{enumerate}
\item Let ($\mathcal{E}$) be an exact sequence of $\mathbb{R}$-vector spaces
\begin{equation*}
0 \to V_0 \xrightarrow{T_0} V_1 \xrightarrow{T_1} ... \xrightarrow{T_{n-1}} V_n \to 0  \quad \quad  (\mathcal{E}).
\end{equation*}
We split ($\mathcal{E}$) into two exact sequences ($\mathcal{E}_1$) and ($\mathcal{E}_2$) such that $\beta\alpha=T_i$. 
\begin{equation*}
0 \to V_0 \xrightarrow{T_0} V_1 \xrightarrow{T_1} ... \xrightarrow{T_{i-1}} V_i \xrightarrow{\alpha} J \to 0  \quad \quad  (\mathcal{E}_1).
\end{equation*}
\begin{equation*}
0 \to J \xrightarrow{\beta} V_{i+1} \xrightarrow{T_{i+1}} ... \xrightarrow{T_{n-1}} V_n \to 0  \quad \quad  (\mathcal{E}_2).
\end{equation*}
Then by an induction argument, we can show that 
$\nu(\mathcal{E})=\nu(\mathcal{E}_1)\nu(\mathcal{E}_2)^{(-1)^{i}}$.
\item Let ($\mathcal{E}^{*}$) be the dual sequence of ($\mathcal{E}$) and let $B_i^{*}$ be the dual basis of $B_i$. Then with respect to $\{B_i^{*}\}$ and $\{B_i\}$, 
$ \nu(\mathcal{E}^{*})=\nu(\mathcal{E})^{-1}$.
\end{enumerate}
\end{rmk}

\begin{lemma}\label{det_3x3}
Consider the following commutative diagram 
\[ \xymatrixrowsep{0.25in}\xymatrix{
      &  0 \ar[d] & 0 \ar[d] & 0 \ar[d] \\
  0 \ar[r] & A_1 \ar[d]^{\theta_1} \ar[r]^{\phi_A}  & A_2  \ar[d]^{\theta_2} \ar[r]^{\psi_A}  &   A_3 \ar[d]^{\theta_3} \ar[r] & 0
 & (\mathcal{E}_A)\\
   0 \ar[r] & B_1 \ar[d]^{\tau_1} \ar[r]^{\phi_B} &   B_2  \ar[d]^{\tau_2} \ar[r]^{\psi_B}  & B_3  \ar[d]^{\tau_3} \ar[r] &  0
 & (\mathcal{E}_B)\\
    0 \ar[r] & C_1 \ar[d] \ar[r]^{\phi_C}  &   C_2 \ar[d]  \ar[r]^{\psi_C}  &   C_3 \ar[d] \ar[r] & 0
 & (\mathcal{E}_C)     \\
  & 0 & 0 & 0 \\
 & (\mathcal{E}_1) & (\mathcal{E}_2) & (\mathcal{E}_3)}
 \]
Let {$\{a_i,b_i,c_i\}_{i=1}^3$} be bases for $\{A_i,B_i,C_i\}_{i=1}^3$. 
Then with respect to these bases
\[ \frac{\nu(\mathcal{E}_2)}{\nu(\mathcal{E}_1)\nu(\mathcal{E}_3)}=\frac{\nu(\mathcal{E}_B)}{\nu(\mathcal{E}_A)\nu(\mathcal{E}_C)}. \]
\end{lemma}
\begin{proof}
By the definition of $\nu(\mathcal{E}_B)$, we have 
\begin{equation}\label{det_3x3_eq1}
 \wedge_{i}b_2^i = \pm \nu(\mathcal{E}_B)^{-1}   (\wedge_{i}\phi_B(b_1^i))\wedge (\wedge_{i} \psi_B^{-1}(b_3^i)).
\end{equation}
Let $M:=(\wedge_{i}\phi_B\theta_1(a_1^i))$ and $N:=(\wedge_{i} \phi_B\tau_1^{-1}(c_1^i))$. By the definition of $\nu(\mathcal{E}_1)$, 
\begin{equation}\label{det_3x3_eq2}
\wedge_{i}\phi_B(b_1^i) = \pm \nu(\mathcal{E}_1)^{-1} (\wedge_{i}\phi_B\theta_1(a_1^i))\wedge (\wedge_{i} \phi_B\tau_1^{-1}(c_1^i)) = \pm \nu(\mathcal{E}_1)^{-1} M \wedge N .
\end{equation}
Let $P:=(\wedge_{i}\psi_B^{-1}\theta_3(a_3^i))$ and $Q:=(\wedge_{i} \psi_B^{-1}\tau_3^{-1}(c_3^i))$. By the definition of
$\nu(\mathcal{E}_3)$, 
\begin{equation}\label{det_3x3_eq3}
\wedge_{i}\psi_B^{-1}(b_3^i) = \pm \nu(\mathcal{E}_3)^{-1} (\wedge_{i}\psi_B^{-1}\theta_3(a_3^i))\wedge (\wedge_{i} \psi_B^{-1}\tau_3^{-1}(c_3^i)) = \pm \nu(\mathcal{E}_3)^{-1} P \wedge Q.
\end{equation}
 Putting together ($\ref{det_3x3_eq1}$), ($\ref{det_3x3_eq2}$) and ($\ref{det_3x3_eq3}$), we deduce
\begin{equation}\label{det_3x3_eq4}
\wedge_{i}b_2^i = \pm \nu(\mathcal{E}_B)^{-1}\nu(\mathcal{E}_1)^{-1}\nu(\mathcal{E}_3)^{-1}
M\wedge N \wedge P \wedge Q.
\end{equation}
Let $M':=(\wedge_{i}\theta_2\phi_A(a_1^i))$, $N':=(\wedge_{i} \tau_2^{-1}\phi_C(c_1^i)) $, 
$P':=(\wedge_{i}\theta_2\psi_A^{-1}(a_3^i)) $,
 and $Q':=(\wedge_{i} \tau_2^{-1}\psi_C^{-1}(c_3^i))$.
By a similar argument, we have 
\begin{equation}\label{det_3x3_eq5}
\wedge_{i}b_2^i = \pm \nu(\mathcal{E}_2)^{-1}\nu(\mathcal{E}_A)^{-1}\nu(\mathcal{E}_C)^{-1}
M' \wedge N' \wedge P' \wedge Q'.
\end{equation}
From ($\ref{det_3x3_eq4}$) and ($\ref{det_3x3_eq5}$), it is enough to show
\[ M\wedge N \wedge P \wedge Q = M' \wedge N' \wedge P' \wedge Q'. \]
Indeed, we have $M=M'$ since $\phi_B\theta_1=\theta_2\phi_A$. Let $x=N-N'$. As $\phi_B\tau_1^{-1}(c_1^i)-\tau_2^{-1}\phi_C(c_1^i) \in \ker\tau_2=\mathrm{im}(\theta_2)$, we deduce $x$ is a finite sum of wedge products such that each product has a factor which is an element of $\mathrm{im}(\theta_2)$.

Similarly, let $y=P-P'$. As
$\psi_B^{-1}\theta_3(a_3^i)) - \theta_2\psi_A^{-1}(a_3^i) \in (\ker\psi_B) = (\mathrm{im}\phi_B)$
, $y$ is a finite sum of wedge products such that each product has a factor belonging to $\mathrm{im}(\phi_B)$.

Since $\tau_3\psi_B=\psi_C\tau_2$, $\psi_B^{-1}\tau_3^{-1}(c_3^i)-\tau_2^{-1}\psi_C^{-1}(c_3^i)$ is an element of $\ker(\tau_3\psi_B)$. As a vector space, $\ker \tau_3\psi_B$ is spanned by $\ker\psi_B=\mathrm{im}\phi_B$ and $\psi_B^{-1}\ker(\tau_3)=\psi_B^{-1}\mathrm{im}(\theta_3)$. 
Therefore 
\[ Q-Q'=(\wedge_{i} \psi_B^{-1}\tau_3^{-1}(c_3^i))- (\wedge_{i} \tau_2^{-1}\psi_C^{-1}(c_3^i)) =z+t
\]
where $z,t$ are finite sums such that each summand of $z$ (respectively $t$) has a factor belonging to $\mathrm{im}(\phi_B)$ (respectively $\psi_B^{-1}\mathrm{im}(\theta_3)$).

{Claim:}
$M' \wedge x \wedge P'=0$,  \quad $M \wedge N \wedge y =0$, \quad $M \wedge N \wedge z =0$, \quad $P \wedge t =0$.

{Proof of claim:}
Recall that $M'=(\wedge_{i}\theta_2\phi_A(a_1^i))$ and $P'=(\wedge_{i}\theta_2\psi_A^{-1}(a_3^i))$.
It is clear that $\{\theta_2\phi_A(a_1^i),\theta_2\psi_A^{-1}(a_3^i)\}$ span $\mathrm{im}(\theta_2)$. As
each summand of $x$ has a factor belonging to $\mathrm{im}(\theta_2)$, $M' \wedge x \wedge P'=0$.
The rest of the claim can be proved in a similar fashion. Finally,
\begin{eqnarray*}
M\wedge N \wedge P \wedge Q &=& M \wedge N \wedge P \wedge (Q'+z+t) =  M \wedge N \wedge P \wedge Q' \\
&=& M \wedge N \wedge (P'+y) \wedge Q' = M \wedge N \wedge P' \wedge Q'  \\
&=& M' \wedge (N'+x) \wedge P' \wedge Q' = M' \wedge N' \wedge P' \wedge Q'.
\end{eqnarray*}
\end{proof}
The following proposition can be deduced from Lemma $\ref{det_3x3}$ by an induction argument (whose proof we omit).
\begin{prop}\label{det_nxn}
Consider the following commutative diagram of $\mathbb{R}$-vector spaces
\begin{equation}\label{det_nxn_sq1}
\xymatrixrowsep{0.19in}\xymatrix{
          &  0  \ar[d]  &   0 \ar[d]  &   0 \ar[d] &   &   0 \ar[d] &     \\
0  \ar[r] & V_{0,0}  \ar[r]^{T_{0,0}} \ar[d]^{T'_{0,0}}  &   V_{0,1}  \ar[r]^{T_{0,1}} \ar[d]^{T'_{0,1}} &  V_{0,2}  \ar[r] \ar[d]^{T'_{0,2}} & \cdots \ar[r]  &   V_{0,n}  \ar[r] \ar[d]^{T'_{0,n}}  & 0  & (\mathcal{R}_0)\\
0  \ar[r] & V_{1,0}  \ar[r]^{T_{1,0}} \ar[d]  &   V_{1,1}  \ar[r]^{T_{1,1}} \ar[d] &  V_{1,2}  \ar[r] \ar[d] & \cdots \ar[r]  &   V_{1,n}  \ar[r] \ar[d]  & 0   & (\mathcal{R}_1)\\
        &   \vdots \ar[d]  &  \vdots  \ar[d] & \vdots \ar[d] & \ddots   &   \vdots \ar[d]  &    \\
0  \ar[r] & V_{m,0}  \ar[r]^{T_{m,0}} \ar[d]  &   V_{m,1}  \ar[r]^{T_{m,1}} \ar[d] &  V_{m,2}  \ar[r] \ar[d] & \cdots \ar[r]  &   V_{m,n}  \ar[r] \ar[d]  & 0   & (\mathcal{R}_m) \\
 &  0   &   0   &   0  &   &   0  &     \\
  &  (\mathcal{C}_0)   &   (\mathcal{C}_1)   &   (\mathcal{C}_{2})  &   &   (\mathcal{C}_n)  &  
}
\end{equation}
Let $B_{i,j}$ be an ordered basis for $V_{i,j}$. Then with respect to the bases $B_{i,j}$,
\begin{equation}\label{det_nxn_eq1}
\prod_{i=0}^{n}\nu(\mathcal{C}_i)^{(-1)^{i}} =  \prod_{i=0}^{m}\nu(\mathcal{R}_i)^{(-1)^{i}}.
\end{equation}
\end{prop}

\begin{cor}\label{det_2x3}
Consider the following commutative diagram with exact rows
\[\xymatrixrowsep{0.25in}\xymatrix{
  0 \ar[r] & A_1 \ar[d]^{\theta_1} \ar[r]^{\phi_A}  & A_2  \ar[d]^{\theta_2} \ar[r]^{\psi_A}  &   A_3 \ar[d]^{\theta_3} \ar[r] & 0 & (\mathcal{E}_A)\\
   0 \ar[r] & B_1  \ar[r]^{\phi_B} &   B_2  \ar[r]^{\psi_B}  & B_3   \ar[r] &  0 & (\mathcal{E}_B)}
 \]
Assume further that all the vertical maps are isomorphisms. Let $\{a_i,b_i\}_{i=1}^2$ be bases for $\{A_i,B_i\}_{i=1}^2$ respectively. Then with respect to these bases
\[ \frac{|\det\theta_1||\det\theta_3|}{|\det\theta_2|}=\frac{\nu(\mathcal{E}_A)}{\nu(\mathcal{E}_B)} .\]
\end{cor}
\subsection{Orders of Torsion Subgroups}
For a finitely generated abelian group $M$, we write $M_f$ for $M/M_{tor}$ and by an integral basis for $M$, we mean a $\mathbb{Z}$-basis for $M_f$. Moreover, if $f:M\to N$ is a group homomorphism then $f_{tor}:M_{tor}\to N_{tor}$.
\begin{lemma}\label{torsion_group}
Consider the following exact sequence of finitely generated abelian groups
\begin{equation}\label{hom1_seq1}
0\to A \to B \xrightarrow{\phi} C \xrightarrow{\psi} D \to E \to 0.
\end{equation}
 Assume $A$ is finite. Then the orders of the torsion subgroups are related by
 \[ \frac{[A][C_{tor}]}{[B_{tor}][D_{tor}]}=\frac{1}{[\mathrm{cok}(\psi_{tor})]}. \]
\end{lemma}

\begin{proof}
This is a consequence of the fact that 
if $A$ is finite then the map $B_{tor} \to (B/A)_{tor}$ is surjective.
\end{proof}

\begin{lemma}\label{det_tor_3term}
Let ($\mathcal{E}$) be an exact sequence of finitely generated abelian groups
\[ 0 \to A \xrightarrow{\phi} B \xrightarrow{\psi} C \to 0 \] 
and $(\mathcal{E})_{\mathbb{R}}$ be the sequence ($\mathcal{E}$) tensoring with $\mathbb{R}$. Then with respect to any integral bases,
\begin{equation}\label{det_tor_3term_eq1}
\nu(\mathcal{E})_{\mathbb{R}}= \frac{[A_{tor}][C_{tor}]}{[B_{tor}]} =  [\mathrm{cok}(\psi_{tor})].
\end{equation}
\end{lemma}

\begin{proof}
From remark $\ref{rmk_determinant}$, for any section $\gamma$ of $\psi_{\mathbb{R}}$,
$ \nu(\mathcal{E})_{\mathbb{R}}=|\det\theta_{\gamma}| $ , with respect to integral bases.
As $|\det\theta_{\gamma}|$ is independent of the choice of integral bases, we only need to show that there exist a section $\gamma$ of $\psi_{\mathbb{R}}$ and integral bases of $A$, $B$ and $C$ such that ($\ref{det_tor_3term_eq1}$) holds.

Consider the following commutative diagram
\[ \xymatrixrowsep{0.2in}\xymatrix{ 0 \ar[r] & B_{tor} \ar[r] \ar[d]^{\psi_{tor}} & B \ar[d]^{\psi} \ar[r] &  B_f \ar[d]^{\psi_{f}} \ar[r] & 0 \\
                    0 \ar[r] & C_{tor} \ar[r]          & C         \ar[r] &  C_f         \ar[r] & 0 \\}
\]
The Snake Lemma yields $\mathrm{cok}(\psi_f)=0$ and 
\[ 0 \to \ker\psi_{tor} \to \ker\psi=\mathrm{im}(\phi) \to \ker\psi_{f} \to 
\mathrm{cok}\psi_{tor} \to  0 .\]
Therefore, $[\mathrm{cok}\psi_{tor}]=[\ker\psi_{f}/\mathrm{im}(\phi)]$. Since $\psi_f : B_f \to C_f$ is surjective and $C_f$ is a free abelian group, there exists a section $\gamma : C_f \to B_f$ of $\psi_f$ and we have 
$ B_f = \ker(\psi_f) \oplus \gamma(C_f)$. Take any integral basis $ \{ w_i\}_{i=1}^s$ for $C_f$. 
By the Smith Normal form, there are $\mathbb{Z}$-bases $\{u_i\}_{i=1}^r$ for $A_f$ and $\{v_i\}_{i=1}^r$ for $\ker\psi_{f}$ such that $\phi_f(u_i)=m_iv_i$ where $m_i$ is a positive integer for $i=1,..,r$. Then $\{u_1,...,u_r,w_1,...,w_s\}$ 
and $\{v_1,...,v_r,\gamma(w_1),...,\gamma(w_s)\}$ form integral bases for $A_{\mathbb{R}}\oplus C_{\mathbb{R}}$ and  $B_{\mathbb{R}}$. Moreover, $[\ker\psi_{f}/\mathrm{im}(\phi)] = \prod_{i=1}^r|m_i|$.

Let $\theta_{\gamma} : A_{\mathbb{R}}\oplus C_{\mathbb{R}} \to B_{\mathbb{R}}$ be given by
$\theta(a,c)=\psi(a)+\gamma(c)$. Then with respect to the above integral bases, $\det(\theta_{\gamma})=\prod_{i=1}^r|m_i|$.
 As a result, 
\[ \nu(\mathcal{E})_{\mathbb{R}}=|\det\theta_{\gamma}| = \prod_{i=1}^r|m_i|
= {\left[\frac{\ker\psi_{f}}{\mathrm{im}(\phi)}\right]} = {[\mathrm{cok}\psi_{tor}]}= \frac{[A_{tor}][C_{tor}]}{[B_{tor}]} .\]
\end{proof}

\begin{prop}\label{det_tor}
Let $(\mathcal{E})$ be an exact sequence of finitely generated abelian groups
\[ 0 \to A_0 \xrightarrow{} A_1 \to ... \to A_n \to 0. \]
Let $(\mathcal{E})_{\mathbb{R}}$ be the sequence $(\mathcal{E})$ tensoring with $\mathbb{R}$. Let $B_i$ be an ordered integral basis for $A_i$. Then with respect to $B_i$,
\begin{equation*}\label{det_tor_eq1}
\nu(\mathcal{E})_{\mathbb{R}} = \prod_{i=0}^{n}[(A_i)_{tor}]^{(-1)^{i}} .
\end{equation*}
\end{prop}
\begin{proof}
The proof uses induction on $n$. The base case when $n=2$ is Lemma $\ref{det_tor_3term}$.
\end{proof}
\begin{cor}\label{det_tor_5term}
Suppose we have an exact sequence of finitely generated abelian groups
\[ 0 \to A \to B \xrightarrow{\phi} C \xrightarrow{\psi} D \to E \to 0 \] 
where $A$ and $E$ are finite groups. Then with respect to integral bases,
\[ \nu([ 0 \to B_{\mathbb{R}} \xrightarrow{\phi} C_{\mathbb{R}} \xrightarrow{\psi} D_{\mathbb{R}} \to 0])
= \frac{[B_{tor}][D_{tor}]}{[A][C_{tor}][E]} = \frac{[\mathrm{cok}(\psi_{tor})]}{[\mathrm{cok}\psi]}. \]
\end{cor}


\bigskip
\begin{align*}
& \large \mbox{Department of Mathematics, University of Regensburg, 93040 Regensburg, Germany.} \\
& \large \mbox{Email : minh-hoang.tran@mathematik.uni-regensburg.de}
\end{align*}

	\end{document}